\documentclass[11pt,twoside]{amsart}

\usepackage{geometry}
\usepackage{hyperref}
\usepackage{amsmath,amssymb,amsthm}
\usepackage[active]{srcltx}
\usepackage[dvips]{graphicx}
\title[Viscosity solutions of elliptic branches]{On viscosity solutions to the Dirichlet problem for elliptic branches of nonhomogeneous fully nonlinear equations}
\author{Marco Cirant}
\address{Dipartimento di Matematica ``F. Enriques''\\ Universit\`a di Milano\\ Via C. Saldini 50\\ 20133--Milano, Italy}
\email{marco.cirant@unimi.it (Marco Cirant)}\thanks{Cirant partially supported by the Fondazione CaRiPaRo Project ``Nonlinear Partial Differential Equations: models, analysis, and control-theoretic problems''}
\author{Kevin R.\ Payne}
\address{Dipartimento di Matematica ``F. Enriques''\\ Universit\`a di Milano\\ Via C. Saldini 50\\ 20133--Milano, Italy}
\email{kevin.payne@unimi.it (Kevin R. Payne)} \thanks{Payne partially supported the Gruppo Nazionale per l'Analisi Matematica, la Probabilit\`a e le loro Applicazioni (GNAMPA) of the Istituto Nazionale di Alta Matematica (INdAM)} \thanks{The authors wish to thank an anonymous referee for the constructive criticism of a previous version of this manuscript.}

\date{\today} 


\subjclass[2010]{}

\catcode`@=11 \@addtoreset{equation}{section} \catcode`@=12

\newcommand{\F}{\mathcal{F}}
\newcommand{\Ss}{\mathcal{S}}

\newcommand{\R}{\mathbb{R}}
\newcommand{\N}{\mathbb{N}}
\newcommand{\Pd}{\mathcal{P}}
\newcommand{\dPd}{\widetilde{\Pd}}
\newcommand{\UC}{\mathrm{UC}}
\newcommand{\USC}{\mathrm{USC}}
\newcommand{\LSC}{\mathrm{LSC}}

\newcommand{\SA}{\mathrm{SA}}

\newcommand{\EC}[1]{\overrightarrow{#1}}

\newcommand{\ThetaD}{\widetilde{\Theta}}

\newcommand{\TSH}{\Theta \mathrm{SH}}
\newcommand{\TSHD}{\widetilde{\Theta} \mathrm{SH}}
\newcommand{\PSHD}{\widetilde{\mathcal{P}} \mathrm{SH}}
\newcommand{\veps}{\varepsilon}

\newtheorem{thm}{\textbf{Theorem}}[section]
\newtheorem{lem}[thm]{\textbf{Lemma}}
\newtheorem{prop}[thm]{\textbf{Proposition}}
\theoremstyle{remark}
\newtheorem{rem}[thm]{\textbf{Remark}}

\newtheorem{exe}[thm]{\textbf{Example}}
\theoremstyle{definition}
\newtheorem{defn}[thm]{\textbf{Definition}}
\newtheoremstyle{Claim}{}{}{\itshape}{}{\itshape\bfseries}{:}{ }{#1}
\theoremstyle{Claim}

\begin{document}

\maketitle

\begin{abstract}
For scalar fully nonlinear partial differential equations $F(x, D^2u(x)) = 0$ with $x \in \Omega \subset \subset \R^N$,  we present a general theory for obtaining comparison principles and well posedness for the associated Dirichlet problem with continuous boundary data. In particular, we treat {\em admissible viscosity solutions} $u$ of {\em elliptic branches} of the equation, where $F(x, \cdot)$ need not be monotone on all of $\Ss(N)$, the space of symmetric $N \times N$ matrices. This is accomplished by exploiting the fundamental ideas of Krylov \cite{Kv95} and following the program initiated by Harvey and Lawson \cite{HL09} in the homogeneous case when $F$ does not depend on $x$. An elliptic branch of the equation is encoded by a set valued map $\Theta$ from  $\Omega$ into the {\em elliptic subsets} of $\Ss(N)$ along which $F$ needs to be monotone and the degenerate elliptic PDE is replaced by the differential inclusion $D^2u(x) \in \partial \Theta(x)$ while subsolutions correspond to $D^2u(x) \in \Theta(x)$. Weak solutions to such differential inclusions are defined by using the notion given in \cite{HL09} in a pointwise manner. If $\Theta$ is uniformly upper semicontinuous, we show that the comparison principle holds for these weak solutions and that Perron's method yields a unique continuous solution to the {\em abstract Dirichlet problem} for $\Theta$-harmonic functions provided that the boundary is suitably convex with respect to the elliptic map $\Theta$ and its dual $\widetilde{\Theta}$ in the sense of Harvey and Lawson. When $\Theta$ encodes an elliptic branch of a given PDE, these $\Theta$-harmonic functions are shown to be admissible viscosity solutions of the PDE problem.  Various applications are described in terms of structural conditions on $F$ which ensure the existence of the needed elliptic map $\Theta$. Examples include {\em non-totally degenerate equations} and equations involving the eigenvalues of the Hessian and their perturbations. In certain situations, the methods employed here will be shown to operate freely, while classical viscosity approaches may not.
\end{abstract}

\setcounter{tocdepth}{1}
\tableofcontents

\section{Introduction}

In this work, we will study the Dirichlet problem for second order fully nonlinear PDE of the form
\begin{equation}\label{FNLN}
F(x, D^2 u(x)) = 0, \ \ x \in \Omega
\end{equation}
\begin{equation}\label{DBCs}
u(x) = \varphi(x), \ \ x \in \partial \Omega
\end{equation}
where $\Omega \subset \R^N$ is a bounded open domain with $C^2$ boundary, $F$ is a continuous function of its arguments and $\varphi$ a given continuous function. In particular, we will study the validity of the comparison principle for \eqref{FNLN} and the well posedness of the Dirichlet problem \eqref{FNLN}-\eqref{DBCs} by way of a Perron method for {\em admissible viscosity solutions} of {\em elliptic branches} of the PDE \eqref{FNLN}. As the stated aims imply, we will be interested in equations for which $F(x,\cdot)$ is not monotone on all of $\Ss(N)$ (the space of symmetric $N \times N$ matrices) and hence the need to select suitable elliptic branches by making use of Krylov's general notion of ellipticity \cite{Kv95}. The branch will be encoded by a function
\begin{equation}\label{BFI}
\Theta: \Omega \to \wp(\Ss(N)) = \{ \Phi: \Phi \subset \Ss(N) \}
 \end{equation}
whose values $\Theta(x)$ must be an {\em elliptic set}; that is, an non empty, closed, and proper subset of $\Ss(N)$ which is stable under sums with non-negative elements in $\Ss(N)$. The PDE will be replaced by the differential inclusion
\begin{equation}\label{DIF}
     \mbox{$D^2u(x) \in \partial \Theta(x)$ for each $x \in \Omega$}
\end{equation}
and $\Theta$ will give an elliptic branch of \eqref{FNLN} provided that
\begin{equation}\label{branch}
    \mbox{$ \partial \Theta(x) \subset \Gamma(x) := \{ A \in \Ss(N): \ F(x,A) = 0 \}$ for each $x \in \Omega$,}
\end{equation}
where obviously one assumes that $\Gamma(x) \neq \emptyset$. When $u$ is merely continuous, the notion of admissible viscosity solution will give a meaning to \eqref{DIF}. We remark that we will often think of the function $\Theta$ in \eqref{BFI} as a {\em set-valued map}; that is, as a multi-valued map
\begin{equation}\label{MVBFI}
    \Theta: \Omega \multimap \Ss(N)
\end{equation}
which is {\em strict} in the sense that $\Theta(x) \neq \emptyset$ for each $x \in \Omega$ (see Chapter 1 of Aubin and Cellina \cite{AC84} for the elementary notions concerning set-valued maps). We will use the same notation $\Theta$ for these two interpretations where the only real difference here concerns how to formulate their regularity properties. 

To carry out this program, we will adapt the recent approach of Harvey and Lawson \cite{HL09} to such questions for homogeneous equations $F(D^2u) = 0$. Our intent is twofold: generalize their approach to allow for explicit $x$-dependence in the equation and strengthen the connections of this approach with standard viscosity theory, which will result in some situations for which our methods operate freely while the classical viscosity methods do not. For example, we will be able to treat directly perturbed Monge-Amp\`ere type equations of the form 
\begin{equation}\label{PMAE}
    \left[ {\rm det}(D^2u(x) + M(x)) \right]^{1/N} = f(x),
\end{equation}
with $M$ a uniformly continuous $\Ss(N)$ valued function and $f$ a non-negative uniformly continuous real valued function. The standard structural conditions which ensure the validity of the comparison principle for viscosity solutions of \eqref{PMAE} may fail (see the conditions \eqref{caba1}-\eqref{caba2} and Example \ref{SC_CAB}). While taking the logarithm of \eqref{PMAE} converts it into a Bellman equation suitable for the viscosity theory, obviously one cannot allow $f$ to vanish on $\Omega$ and we will require only that $f \geq 0$ (see Remark \ref{PMA2}). Moreover, we can also treat certain Bellman operators with milder regularity assumptions (see Remark \ref{linear_eqs} and Example \ref{exe:linear}). With these motivations in mind, we will describe the results we obtain and make additional comparisons with the related literature.

We begin with a brief discussion of the method initiated by Harvey and Lawson in \cite{HL09} and how it will be generalized here. Following Krylov's lead, they shift attention from the particular PDE \eqref{FNLN} to the differential inclusion \eqref{DIF} and introduce a notion of weak solutions (as well as weak subsolutions and supersolutions) in terms of the {\em elliptic map} $\Theta$ of \eqref{BFI}, taking values in the elliptic subsets $\mathcal{E}$ of $\Ss(N)$. In \cite{HL09}, this map is constant. For $u \in C^2(\Omega)$, being {\em $\Theta$-subharmonic in $x$} just means
\begin{equation}\label{SHR}
    D^2u(x) \in \Theta(x).
\end{equation}
For $u$ merely upper semicontinuous on $\Omega$, being $\Theta$-subharmonic in $x$ requires that $u + v$ be {\em subaffine in $x$} for each $v \in C^2(\Omega)$ which belongs to a natural space of test functions. These test functions are $\widetilde{\Theta}$-subharmonic in $x$ where $\widetilde{\Theta}(x)$ is the {\em dual set} which is defined in a set theoretic way (see \eqref{DM}) and is also an elliptic set for each $x$. Subaffinity in $x$ means the validity of a weak comparison principle with respect to affine functions on each sufficiently small neighborhood of $x$ (see Definition \ref{def:Weak_SH} also for the definitions of $\Theta$-superharmonic and $\Theta$-harmonic functions). In this way, one can treat the validity of a comparison principle and attempt a Perron method for the {\em abstract} Dirichlet problem involving $\Theta$-harmonic functions for any elliptic map $\Theta$ and not necessarily one that comes from any particular PDE such as \eqref{FNLN}.

In the homogeneous case in which $\Theta$ is constant, \cite{HL09} proves the validity of the comparison principle for these weak solutions and obtains well-posedness of the Dirichlet problem for domains which are suitably convex with respect to $\Theta$ and $\widetilde{\Theta}$, which ensures the existence of suitable upper and lower solutions (barriers) needed for the Perron scheme. This convexity is expressed in terms of global defining functions for the boundary $\partial \Omega$ which are $C^2$ strictly $\EC{\Theta}$, $\EC{\widetilde{\Theta}}$-subharmonic, where $\EC{\Theta}$ and $\EC{\widetilde{\Theta}}$ are {\em elliptic cones} defined by the elliptic maps (see \eqref{defn:ass_cone}). Here, in the non homogeneous case in which $\Theta$ is a non-constant elliptic map, we prove the validity of the comparison principle provided that $\Theta$ thought of as a multi-valued map \eqref{MVBFI} is {\em uniformly upper semicontinuous} (see Theorem \ref{thm:CP}). The proof is an adaptation of the argument used in \cite{HL09} which makes use of the so-called subaffine theorem (see Theorem \ref{thm:SAT}), which we show to hold under the regularity property of uniform upper semicontinuity. This regularity property means that for every $\veps > 0$ there exists $\delta = \delta(\veps)$ such that
\begin{equation}\label{UUSC_def}
  \mbox{$x,y \in \Omega$ with $|x-y| < \delta \ \Rightarrow \ \Theta(x) + \veps I \subset \Theta(y)$ \ \ and \ \ $\Theta(y) + \veps I \subset \Theta(x)$.}
\end{equation}
For elliptic maps $\Theta$, we show that the regularity property \eqref{UUSC_def} has many useful consequences. It is passed on to the dual map $\widetilde{\Theta}$ in Proposition \ref{uuscd} and in Proposition \ref{UCHD} is shown to be equivalent to the {\em uniform continuity} of $\Theta$ in the interpretation \eqref{BFI} if one places the metric topology coming from the Hausdorff distance on the closed subsets of $\Ss(N)$. This characterization allows one to extend such maps to uniformly upper semicontinuous maps on all of $\overline{\Omega}$ (see Proposition \ref{extension}). This is useful for the PDE applications as the needed structural conditions on $F$ can be stated with reference to the open domain $\Omega$. The regularity property \eqref{UUSC_def} also implies that the natural {\em elliptic cone maps} used to describe the boundary convexity for Perron's method are constant maps (see Proposition \ref{uusc_cone}). Finally, under the regularity property \eqref{UUSC_def}, in Theorem \ref{thm:EU} we implement Perron's method to find the well-posedness of the abstract Dirichlet problem for (weakly) $\Theta$-harmonic functions with prescribed continuous boundary values, including a description of what constitutes an admissible domain in terms of the needed convexity which ensures the existence of suitable barriers.

We now describe the connections of this abstract theory to viscosity solutions of elliptic branches of \eqref{FNLN} and the resulting applications. What needs to be done is to describe when a given PDE will admit an elliptic branch corresponding to a uniformly upper semicontinuous map $\Theta$ and show that the weak solutions associated to $\Theta$ are viscosity solutions of \eqref{FNLN}. To this end, we formalize a definition of {\em $\Phi$-admissible viscosity solutions} to \eqref{FNLN} where $\Phi: \Omega \to \wp(\Ss(N))$ is any elliptic map along which $F$ increases and for which $\Phi(x)$ intersects the zero locus $\Gamma(x)$ defined in \eqref{branch} (see Defintion \ref{Vs_def}). In Proposition \ref{pick_branch}, we give sufficient conditions on $F$ so that
\begin{equation}\label{PBE}
    \Theta(x):= \{ A \in \Phi: F(x,A) \geq 0 \}
\end{equation}
defines an elliptic branch of \eqref{FNLN}. The ellipticity of $\Theta$ will be automatic, while the branch condition \eqref{def_branch2} which implies \eqref{branch} needs to be checked in the applications. We study the equivalence between $\Phi$-admissible viscosity solutions to $F(x, D^2u) = 0$ and $\Theta$-harmonic functions for $\Theta$ defined by \eqref{PBE} in Proposition \ref{SHCVS}.

At this point, it remains only to find a structural condition on $F$ that ensures that $\Theta$ is uniformly upper semicontinuous. One such condition is given in Proposition \ref{UCbranch} and it asks that for all small $\veps > 0$ there exists $\delta = \delta(\veps)$ such that
\begin{equation}\label{UCFI}
F(y,A + \veps I) \geq F(x,A) \quad \forall A \in \Phi(x), \forall x,y \in \Omega \ {\rm such \ that \ } |x - y| < \delta.
\end{equation}
In the special case $F(x,A) = G(A) - f(x)$ with $G$ increasing along $\Phi$ and $f$ uniformly continuous, the condition \eqref{UCFI} is satisfied provided that $G$ is {\em non-totally degenerate} in some sense (see Example \ref{exe:FGf}). Additional examples are studied in Section \ref{examples} which include equations involving the eigenvalues of the Hessian $\lambda_k(D^2u)$ and their perturbations such as equation \eqref{PMAE}. Finally, it is noted that linear equations do not fit naturally into this theory as unbounded Hessians on the elliptic sets created problems, but they can be forced to fit in rather convoluted way by truncating the equation with well chosen Bellman operators. This is sketched at the end of Section \ref{examples}.

We conclude this introduction with a few additional remarks and comparisons with the literature. A first crucial point is that we have not assumed that $F(x,A)$ is monotone in $A$ on all of $\Ss(N)$. This cuts us off from a large portion of the classical viscosity literature such as summarized in Crandall-Ishii-Lions \cite{CrIsLiPL92} and this lack of global monotonicity is the reason that one considers the elliptic branches in the sense of Krylov. As noted above, classical viscosity approaches have been employed without monotonicity on all of $\Ss(N)$, but they are typically ad hoc approaches which exploit the particular structure of  special but important classes of equations, beginning with those of Monge-Amp\`ere type as treated in Section V.3 of  Ishii-Lions \cite{IsLiPL90} or for prescribed curvature equations as treated in the paper of Trudinger \cite{Tr90}. The approach of Harvey and Lawson which is followed here aims at a general theory when $F$ is not monotone in $A$ on all of $\Ss(N)$.

A second crucial point here is that we consider equations $F(x, D^2u)=0$ without explicit dependence on $u$ and its gradient $Du$. This has several consequences. In the proof of the comparison principle, one does not need to appeal to the so-called Theorem on Sums which is needed in the classical scheme of doubling variables and then penalizing, as introduced in Ishii \cite{Is89}. The subaffine theorem takes its place in the Harvey-Lawson scheme \cite{HL09}. In particular, in their subsequent investigations such as \cite{HL11}, when they add in dependence on $(u, Du)$ and consider equations on Riemannian manifolds, the lack of a maximum principle for subaffine functions in that context leads them to make use of the doubling variables-penalizing scheme and the Theorem on Sums. Hence if one attempts to apply directly the generalizations in \cite{HL11} in our situation, one will need the classical structural conditions. In this sense, we stay close to the original ideas of \cite{HL09} by being able to use their duality to define the relevant space of test functions and the subaffine theorem reduces to the case of semi-convex functions by way of the {\em sup-convolution} and the critical step in the semi-convex case is treated by way of {\em Slodkowski's largest eigenvalue theorem}, which takes the place of {\em Jensen's lemma} in the classical approach (see the discussion after Lemma \ref{SA_ae}).

Other classes of equations for which the comparison principle can be proven without the Theorem on Sums scheme include $F(Du,D^2u)=0$, as studied by Barles-Busca \cite{BaBu11}. This class is clearly complementary to the one considered here and assumes the monotonicity in $A$ on all of $\Ss(N)$. It does share the feature however of not having explicit $u$ dependence which means that one can never have strict monotonicity in $u$, which is known to temper the demands on ellipticity for the validity of the comparison principle. To compensate for the lack of strict monotonicity in $u$, something akin to non-total degeneracy is requested in \eqref{UCFI}, as explained above. Moreover, \eqref{UCFI} appears to rule out the need for Lipschitz regularity conditions in $x$, at least in the linear case as noted in Remark \ref{linear_eqs}. Other recent attempts to address the comparison principle in the absence of strict monotonicity in $u$ include Kawohl-Kutev \cite{KaKu07} and Luo-Eberhard \cite{LoEb05}, where again monotonicity on all of $\Ss(N)$ is used. Moreover, ignoring that difference for a moment, the structural condition (1.6) of \cite{KaKu07} cannot be satisfied for equations independent of $(u,Du)$. The structural conditions of \cite{LoEb05} see more similar to ours, but again the presence of $(u,Du)$ again plays a role.

There is of course also an extensive literature for treating fully nonlinear equations of the form \eqref{FNLN} by means other than viscosity techniques. This is particularly true for special classes of equations such as those of Monge-Amp\`ere type. For example, maximum principles and the continuity method are employed in Section 8 of Caffarelli-Nirenberg-Spruck \cite{CNS84} to treat smooth solutions of equations such as the perturbed Monge-Amp\`ere equation \eqref{PMAE} with $f > 0$. There is also an extensive literature on equations which can be put into Bellman or Bellman-Isaacs form. In particular, Krylov \cite{Kv87} uses this approach to treat strong solutions of \eqref{PMAE} on balls if $f \geq 0$ and $M$ are $C^2$ (see Example 8.2.4 and the comments in Section 8.9 of \cite{Kv87}). Moreover, the solutions can be shown to be $C^{3,\alpha}$ up to the boundary if $f > 0$.  A key point is that the solutions should be semi-convex, lying on the elliptic branch $D^2 u + M \geq 0$ which is the branch we also consider. Finally, one should consider Krylov's pioneering paper \cite{Kv95} (see also the related \cite{Kv94}) on elliptic branches in which well-posedness results for the Dirichlet  problem for equations involving elementary symmetric polynomials in the eigenvalues of $D^2u$ are shown. An important point there is that when the elliptic sets $\Theta(x)$ (or their complements) are convex, then canonical forms of the elliptic branch can be converted into Bellman form, to which solvability results for general nonlinear PDE by barrier techniques apply. We make no assumption on the convexity of $\Theta(x)$. Examples such as the prescribed $k$-th eigenvalue equation of Theorem \ref{thm:PkthE} have branches which are not convex. 

\section{Viscosity solutions of elliptic branches}\label{viscosity_solns}

In this section we will review how to associate an elliptic branch in the sense of Krylov \cite{Kv95} to a fully nonlinear PDE of the form $F(x, D^2u) = 0$ and show how the classical notion of viscosity solutions (and subsolutions) can be expressed in terms of a certain set valued map $\Theta$ taking values in the so-called {\em elliptic subsets} $\mathcal{E}$ of $\mathcal{S}(N)$. The key concepts are the notion of {\em duality} introduced by Harvey and Lawson \cite{HL09} for homogeneous equations (which will be used pointwise in $x$ to define a dual map $\widetilde{\Theta}$)  and the role of {\em subaffine} functions (which satisfy a weak comparison principle with respect to affine functions).

In all that follows, $\mathcal{S}(N)$ denotes the space of symmetric $N \times N$ matrices, which carries the usual partial ordering of the associated quadratic forms and $\lambda_1(A) \leq \cdots \leq \lambda_N(A)$ denote the eigenvalues of $A \in \Ss(N)$. We will denote by $\wp(\Ss(N)):= \{ \Phi: \Phi \subset \Ss(N) \}$ and use the notations $\overline{\Phi}, \Phi^{\circ}$ and $\Phi^c$ for the closure, interior and complement of $\Phi \in \wp(\Ss(N))$.

\subsection{Elliptic branches and elliptic maps}\label{branches}

Let $\Omega \subset \subset \R^N$ and $F: \Omega \times \Ss(N)$ be continuous and consider the fully nonlinear PDE
\begin{equation}\label{FNE}
F( x, D^2 u(x) ) = 0, \quad x \in \Omega .
\end{equation}
The starting point for Krylov's geometric approach is to shift attention from the particular $F$ defining the PDE \eqref{FNE} to its zero locus
$$
\mbox{$\Gamma(x) = \{A \in \Ss(N) : F(x, A) = 0\}$ \quad with \quad $x \in \Omega,$}
$$
which is closed since $F$ is continuous and assumed to be non empty for each $x \in \Omega$. As is well described in \cite{Kv95}, there is much flexibility to be gained by this point of view as many different functions $F$ will give rise to the same collection $\{\Gamma(x)\}_{x \in \Omega}$. Examples of this flexibility for equations of the form \eqref{FNE} will be given in Section \ref{examples}. A {\em branch} of the PDE \eqref{FNE} is determined by the choice of a function
$$
    \Theta : \Omega \to \wp(\Ss(N))
$$
such that
\begin{equation}\label{bdy_Theta}
\partial \Theta(x) \subset \Gamma(x), \quad x \in \Omega
\end{equation}
and the PDE \eqref{FNE} is replaced by the differential inclusion
\begin{equation}\label{diff_inclusion}
D^2 u(x) \in \partial \Theta, \quad  x \in \Omega.
\end{equation}
The meaning of \eqref{diff_inclusion} is clear if $u\in C^2(\Omega)$ but less so if $u$ is just continuous. The {\em ellipticity} of a branch is defined in terms of the map $\Theta$ which must take values in the {\em elliptic subsets} of $\Ss(N)$
\begin{equation}\label{ell_subsets}
\mbox{ $\mathcal{E} := \left\{ \Phi \subset \Ss(N): \Phi \right.$ is closed, non empty, proper and
    $\left. A + P \in \Phi, \ \forall A \in \Phi, P \in \mathcal{P} \right\} $,}
\end{equation}
where
\begin{equation}\label{pos_subsets}
\mathcal{P}:= \{ P \in \Ss(N): P \geq 0 \} = \{ P \in \Ss(N): \lambda_1(P) \geq 0 \},
\end{equation}
The elliptic subsets here were called {\em Dirichlet sets} in \cite{HL09} and an exhaustive list of their elementary properties properties can be found in Section 3 of their paper \cite{HL09}. We have borrowed the term elliptic subset from Krylov \cite{Kv95}, but one should note that his elliptic subsets are open subsets.

\begin{defn}\label{ellipticity} Let $\Theta: \Omega \to \wp(\Ss(N))$.
\begin{itemize}
\item[(a)] $\Theta$ is said to be {\em elliptic at $x \in \Omega$} if $\Theta(x) \in \mathcal{E}$; that is, if
\begin{equation}\label{elliptic1}
\mbox{ $\Theta(x)$ is a closed, non empty and proper subset of $\Ss(N)$}
\end{equation}
and
\begin{equation}\label{elliptic2}
\Theta(x) + \mathcal{P} \subset \Theta(x).
\end{equation}
$\Theta$ is said to be an {\em elliptic map} \footnote{Thinking of $\Theta: \Omega \multimap \Ss(N)$ as a multi-valued map, ellipticity just means that $\Theta(x) \in \mathcal{E}, \forall \ x \in \Omega$.} if this holds for each $x \in \Omega$; that is, if $\Theta: \Omega \to \mathcal{E}$.
\item[(b)] The differential inclusion \eqref{diff_inclusion} is said to be {\em elliptic} if $\Theta$ is an elliptic map.
\item[(c)] An {\em elliptic branch} of \eqref{FNE} is defined by an elliptic map $\Theta$ for which \eqref{bdy_Theta} holds; that is,
$$
\partial \Theta(x) \subset \Gamma(x), \quad x \in \Omega.
$$
\end{itemize}
\end{defn}
Obviously the {\em positivity condition} \eqref{elliptic2} placed on the elliptic map $\Theta$ is related to the familiar condition of {\em degenerate ellipticity} of the PDE \eqref{FNE} in the sense of monotonicity of $F(x,A)$ with respect to $A$. This leads to the first class of general examples of equations \eqref{FNE} for which one can select an elliptic branch. 

\begin{prop}\label{pick_branch} Let $\Phi$ be an elliptic map on $\Omega$ and let $F \in C(\Omega \times \Ss(N), \R)$ be such that $F$ restricted to $\Phi$ is degenerate elliptic; that is,
\begin{equation}\label{deg_ell}
F(x,A) \geq F(x,B), \quad \forall x \in \Omega, \forall A, B \in \Phi(x) \ {\rm such \ that \ } A \geq B.
\end{equation}
and such that the zero locus $\Gamma(x)$ intersects $\Phi(x)$ for each $x \in \Omega$; that is,
\begin{equation}\label{F_strict}
\mbox{ $\forall x \in \Omega$ there exists $A \in \Phi(x)$ such that $F(x,A) = 0.$}
\end{equation}
Then the map $\Theta: \Omega \to \wp(\Ss(N))$ defined by
\begin{equation}\label{def_branch1}
\Theta(x) := \{ A \in \Phi(x):  \ F(x,A) \geq 0 \}
\end{equation}
is an elliptic map and defines an elliptic branch of \eqref{FNE} if, in addition, one has
\begin{equation}\label{def_branch2}
\mbox{ $\partial \Theta(x) \subset  \{ A \in \Ss(N):  F(x,A) \leq 0 \}$ for each $x \in \Omega$.}
\end{equation}
\end{prop}

\begin{proof} For each $x \in \Omega$, $\Theta(x) \neq \emptyset$ by \eqref{F_strict} and is not all of $\Ss(N)$ since $\Phi(x) \in \mathcal{E}$ is a proper subset by the definition \eqref{ell_subsets}. Moreover $\Theta(x)$ is closed since $\Phi(x) \in \mathcal{E}$ is closed and $F$ is continuous, where it would suffice to have $F(x, \cdot)$ upper semicontinuous for each $x$ fixed. Hence one has property \eqref{elliptic1} of an elliptic map. For the positivity condition \eqref{elliptic2}, notice that for each $x \in \Omega$ and for each $A \in \Theta(x) \subset \Phi(x)$ one has
$$
A + P \in\Phi, \quad \forall P \in \mathcal{P}
$$
and hence by \eqref{deg_ell} and \eqref{def_branch1} one has
$$
F(x, A + P) \geq F(x,A) \geq 0.
$$
Hence $\Theta$ is an elliptic map.

    This elliptic map will define an elliptic branch if
\begin{equation}\label{EBPhi}
\partial \Theta(x) \subset \Gamma(x) \cap \Phi(x), \quad x \in \Omega
\end{equation}
which is the condition \eqref{bdy_Theta} restricted to $\Phi$. One easily checks that
\begin{equation}\label{EBPhi2}
\partial \Theta(x) = \left[ \partial \Phi(x) \cap \{ A \in \Ss(N): F(x,A) \geq 0 \}\right] \bigcup \left[ \Phi(x) \cap \{ A \in \Ss(N): F(x,A) = 0 \}\right],
\end{equation}
which yields \eqref{EBPhi} if \eqref{def_branch2} holds since $\Phi(x)$ is closed by definition.
\end{proof}

A simple example is given by $F(x,A) = {\rm det}\,(A)$ and one could choose $\Phi \equiv \mathcal{P}$ so that
$$
\Theta(x) = \{ A \in \mathcal{P}: {\rm det}\,(A) \geq 0 \} = \mathcal{P}, \quad x \in \Omega.
$$
Of course such a case is exactly the kind of homogeneous equation treated in \cite{HL09}. Nonhomogeneous variants and other examples will be discussed in Section \ref{examples}. We remark that the elliptic branch condition \eqref{def_branch2} is not automatic, but it is satisfied in many examples. Finally, we remark that in certain situations, we can allow $\Phi(x) = \Ss(N)$ in Proposition \ref{pick_branch} and in the subsequent theory. Such will be the case for the applications in Theorem \ref{thm:PkthE} and Example \ref{exe:linear}. The point being that even if $F(x,\cdot)$ is monotone on all of $\Ss(N)$, the map $\Theta$ defined by \eqref{def_branch1} may be a proper subset. 

\subsection{Subsolutions of elliptic branches and elliptic maps}\label{subsolns}

We now discuss a notion of weak subsolution of an elliptic branch of the PDE \eqref{FNE} which will lead to the notion of weak supersolutions and weak solutions. We will show that these notions are equivalent to those from the classical theory of viscosity solutions. Following closely the approach of \cite{HL09} for homogeneous equations $F(D^2u) = 0$, the main point will be to define an appropriate notion of an upper semicontinuous subsolution to the differential inclusion \eqref{diff_inclusion}; that is, to interpret weakly the differential inclusion
\begin{equation}\label{sub_diff_incl}
D^2u(x) \in \Theta(x), \quad x \in \Omega
\end{equation}
where $\Theta$ is an (abstract) elliptic map, which need not be a priori linked to any particular PDE \eqref{FNE}. For $u$ with classical regularity, the notion is obvious.

\begin{defn}\label{SHReg} Let $\Theta$ be an elliptic map and $u \in C^2(\Omega)$. One says that
$u$ is \emph{$\Theta$-subharmonic on $\Omega$} if
\begin{equation}\label{ThetaSH_class}
D^2 u(x) \in \Theta(x) \quad \forall x \in \Omega .
\end{equation}
Moreover, $u$ is \emph{strictly $\Theta$-subharmonic on $\Omega$} if
\begin{equation}\label{Strict_ThetaSH_class}
D^2 u(x) \in [\Theta(x)]^{\circ} \quad \forall x \in \Omega.
\end{equation}
\end{defn}
We recall that by definition the set $\Theta(x)$ is closed in the natural topology of $\Ss(N)$ for each $x \in \Omega$. Moreover, each $\Theta(x) \in \mathcal{E}$ has non empty interior. In fact, each $A \in \Theta(x)$ can be written as the limit as $\veps \to 0^+$ of $A + \veps I \in \Theta(x) + \mathcal{P}^{\circ}$. In particular, one has
\begin{equation}\label{closure_interior}
\Theta = \overline{\Theta^{\circ}} \quad \text{for each} \ \Theta \in \mathcal{E}.
\end{equation}
In order to extend this definition to upper semicontinuous functions
$$
    \USC(\Omega) = \{ u: \Omega \to [-\infty, \infty): \ u(x_0) \geq \limsup_{x \to x_0} u(x),\ \forall \ x_0 \in \Omega\},
$$
and to have a notion of weak supersolution for lower semicontinuous functions
$$
    \LSC(\Omega) = \{ u: \Omega \to (-\infty, \infty]: \ u(x_0) \leq \liminf_{x \to x_0} u(x),\ \forall \ x_0 \in \Omega\},
$$
we will find it useful to formulate a definition of viscosity solutions with an {\em admissibility condition}\footnote{Although in slightly different forms, the term admissibility has been used previously in situations where $F$ may not be monotone on all of $\Ss(N)$; for example, see the introduction in Trudinger's study of prescribed curvature equations  \cite{Tr90} and Remark 2.25 of Krylov \cite{Kv95} regarding elliptic branches. } that will correspond to the attempt to select an elliptic branch when $F$ fails to be degenerate elliptic on all of $\Omega \times \Ss(N)$. We recall that the monotonicity in $A$ of $F(x,A)$ is used to show the {\em coherence property} that regular solutions will be also solutions in the viscosity sense.
\begin{defn}\label{Vs_def} Let $F: \Omega \times \Ss(N) \to \R$ be continuous and $\Phi: \Omega \to \mathcal{E}$ an elliptic map.
\begin{itemize}
\item[(a)] One says that $u \in \USC(\Omega) $ is a {\em $\Phi$-admissible viscosity subsolution of \eqref{FNE} in $\Omega$} if for every $x_0 \in \Omega$ and for each $\varphi \in C^2(\Omega)$ one has
\begin{equation}\label{Vss}
\mbox{ $ u - \varphi$ has a local maximum in $x_0 \Rightarrow F(x_0, D^2 \varphi(x_0)) \geq 0$ \ \ and \ \ $D^2 \varphi(x_0) \in \Phi(x_0)$,}
\end{equation}
where it is enough to consider $\varphi$ a quadratic function such that $\varphi(x_0) = u(x_0)$.
\item[(b)] One says that $u \in \LSC(\Omega) $ is a {\em $\Phi$-admissible viscosity supersolution of \eqref{FNE} in $\Omega$} if for every $x_0 \in \Omega$ and for each $\varphi \in C^2(\Omega)$ one has
\begin{equation}\label{VSs}
\mbox{ $ u - \varphi$ has a local minumum in $x_0 \Rightarrow F(x_0, D^2 \varphi(x_0)) \leq 0$ \ \  or \ \ $D^2 \varphi(x_0) \not\in \Phi(x_0)$,}
\end{equation}
where it is enough to consider $\varphi$ a quadratic function such that $\varphi(x_0) = u(x_0)$.
\end{itemize}
One says that $u \in C(\Omega)$ a {\em $\Phi$-admissible viscosity solution of \eqref{FNE} in $\Omega$} if both conditions (a) and (b) hold.
\end{defn}
Notice that if one allows $\Phi(x_0) = \Ss(N)$, then the condition $D^2 \varphi(x_0) \in \Phi(x_0)$ in \eqref{Vss} is automatically satisfied while the condition $D^2 \varphi(x_0) \not\in \Phi(x_0)$ in \eqref{VSs} is vacuous and one recovers the usual notions of viscosity sub and supersolutions (without restrictions). Furthermore, the condition $D^2 \varphi(x_0) \not\in \Phi(x_0)$ in \eqref{VSs} is natural as explained in Section V.3 of Ishii-Lions \cite{IsLiPL90} in the special case of $\Phi \equiv \mathcal{P}$ for equations of Monge-Amp\`ere type. When $F$ and $\Phi$ satisfy \eqref{deg_ell} and \eqref{F_strict}, it is easy to see that $u \in C^2(\Omega)$ will be a $\Phi$-admissible viscosity solution of \eqref{FNE} if
\begin{equation}\label{Phi_AVS}
\mbox{ $F(x, D^2 u(x)) = 0$ \ \ and \ \ $D^2 u(x_0) \in \Phi(x)$ \ \ for each \ $x \in \Omega$.}
\end{equation}
Under mild additional conditions, the converse is true (see the remark following Proposition \ref{coherence}).

If $u$ is \underline{not} of class $C^2$ near $x_0$, the condition \eqref{Vss} says that the failure of the strong maximum principle for $u - \varphi$ in some neighborhood of $x_0$ implies that the regular test function $\varphi$ must satisfy the differential inclusion
$$
    D^2 \varphi(x_0) \in \{ A \in \Phi(x_0): F(x_0, A) \geq 0 \}.
$$
Turning this around, one might ask: {\em if a regular test function $\varphi$ satisfies an appropiate differential inclusion must $u - \varphi$ satisfy some kind of maximum or comparison principle (which is independent of the regularity of $u$)?} The answer is yes, and is the basis of Harvey and Lawson's definition of weak subsolutions to \eqref{diff_inclusion}.  It is formulated for the sum $u + v$ where $v$ is a regular subsolution of an associated {\em dual map} $\widetilde{\Theta}$ and the comparison principle is with respect to affine functions; that is, $u + v$ must be {\em subaffine}. We proceed to describe these two needed notions.

The notion of the dual map is the following.
\begin{defn}\label{dual_map}
Let $\Theta$ be an elliptic map on $\Omega$. The {\em dual map} $\widetilde{\Theta}: \Omega \rightarrow \wp(\Ss(N))$
is defined by
\begin{equation}\label{DM}
\widetilde{\Theta}(x) = \left[  - \Theta(x)^{\circ} \right]^c = - \left[ \Theta(x)^{\circ} \right]^c.
\end{equation}
\end{defn}
We remark that we are merely taking the pointwise dual of the elliptic subset $\Theta(x)$ as defined in \cite{HL09}. Hence the dual map will inherit a number of known properties of the dual of an elliptic subset, such as {\em reflexivity}
\begin{equation}\label{reflexivity}
\widetilde{\widetilde{\Theta}} = \Theta
\end{equation}
and {\em conservation of ellipticity}
\begin{equation}\label{preserve_ell}
\mbox{ $\widetilde{\Theta}$ is elliptic $\Leftrightarrow \Theta$ is elliptic.}
\end{equation}
Indeed, \eqref{reflexivity} is a consequence of the aforementioned property \eqref{closure_interior}. For the property \eqref{preserve_ell}, notice that $\widetilde{\Theta}(x)$ will be closed, non empty and proper if $\Theta(x)$ is. It is not difficult then to show that the positivity property \eqref{elliptic2} passes from $\Theta$ to $\widetilde{\Theta}$ (see Section 4 of \cite{HL09}) and hence $\widetilde{\Theta}$ is elliptic if $\Theta$ is. The reflexivity \eqref{reflexivity} shows that the converse is true as well. A particularly important example of this duality concerns the constant elliptic map $\Theta(x) = \mathcal{P}$ for each $x \in \Omega$ where the dual satisfies
\begin{equation}\label{nonnegative}
\widetilde{\mathcal{P}} = \{ A \in \Ss(N): \lambda_N(A) \geq 0 \}.
\end{equation}
An important relation between an elliptic map and its dual map is the following property
\begin{equation}\label{sum_duals}
A \in \Theta(x), B \in \widetilde{\Theta}(x) \ \ \Rightarrow \ \ A + B \in \widetilde{\mathcal{P}},
\end{equation}
as follows from Lemma 4.3 of \cite{HL09}. This property is crucial to the understanding of various notions in the sequel. In particular, it suggests the definition of weakly $\Theta$-subharmonic functions (see formula \eqref{Weak_SH} of Definition \ref{def:Weak_SH}) in light of the {\em coherence property} of Proposition \ref{coherence} and the characterization \eqref{sat3} of subaffine functions. For regular functions, the property \eqref{sum_duals} is also the infinitesimal version of the so-called {\em subaffine theorem} (see Theorem \ref{thm:SAT}), which plays a key role in the {\em comparison principle} of Theorem \ref{thm:CP}. 

The notion of subaffinity is the following.
\begin{defn}\label{SA_defn} Let $\Omega \subset \R^N$ be an open set. A function $u \in \USC(\Omega)$ is said to be {\em subaffine on $\Omega$} if for every compact $K \subset \Omega$ and every affine function $a$
\begin{equation}\label{SACP}
u \leq a \text{ on } \partial K \Rightarrow u \leq a \text{ on } K.
\end{equation}
The set of all subaffine functions will be denoted by $\SA(\Omega)$.
\end{defn}

This notion can be described in a pointwise fashion by considering what happens if \eqref{SACP} fails. We record the following characterization of $\SA(\Omega)$ which is proven in Lemma 2.2 of \cite{HL09}.

\begin{lem}\label{lemma:nonSA} If $u \in \USC(\Omega)$ then $u$ fails to be subaffine on $\Omega$ if and only if for some point $x_0 \in \Omega$ there exists a triple $(\veps, r, a)$ with $\veps, r > 0$ and $a$ an affine function such that
\begin{equation}\label{e:nonSA}
\mbox{ $(u - a)(x_0) = 0$ \quad and \quad $(u-a)(x) \leq -\veps |x-x_0|^2, \ \forall \ x \in B_r(x_0)$.}
\end{equation}
\end{lem}
We remark that if such a triple exists, then \eqref{e:nonSA} continues to hold for each $\veps' \in (0, \veps]$ and each $r' \in (0, r]$. Hence the notion depends on an arbitrarily small neighborhood of the point $x_0 \in \Omega$ and the quadratic function in \eqref{e:nonSA} has arbitrarily small opening $\veps$. This suggests the following definition \footnote{We are merely making explicit what was implicit in \cite{HL09} and we will find this pointwise definition convenient for comparison with the viscosity theory.}. 

\begin{defn}\label{SAx0_defn} For $x_0 \in \Omega$, define
$$
\SA(x_0) := \left\{ u \in \USC(\Omega): \  \not\exists \ (\veps, r, a) \ {\rm such \ that \ } \eqref{e:nonSA} \ {\rm holds} \right\}.
$$
\end{defn}
Notice that Lemma \ref{lemma:nonSA} says that for $u \in \USC(\Omega)$ one has
$$
    u \in \SA(\Omega) \Leftrightarrow u \in \SA(x_0), \forall \ x_0 \in \Omega.
$$

Combining the notions of the dual map and subaffinity one can define the notions of weak solutions of the differential inclusions \eqref{sub_diff_incl} and \eqref{diff_inclusion}.

\begin{defn}\label{def:Weak_SH} Let $\Theta$ be an elliptic map on $\Omega$.
\begin{itemize}
\item[(a)] A function $u \in \USC(\Omega)$ is said to be \emph{$\Theta$-subharmonic in $x_0 \in \Omega$} if
\begin{equation}\label{Weak_SH}
\mbox{$u + v \in \SA(x_0)$ \quad for all $v \in C^2(\Omega)$ such that $D^2v(x_0) \in \widetilde{\Theta}(x_0)$}
 \end{equation}
and will said to be \emph{$\Theta$-subharmonic on $\Omega$} if \eqref{Weak_SH} holds for each $x_0 \in \Omega$. The spaces of such $\Theta$-subharmonic functions will be denoted by $\TSH(x_0)$ and $\TSH(\Omega)$ respectively.
\item[(b)] A function $u \in \LSC(\Omega)$ is said to be \emph{$\Theta$-superharmonic in $x_0 \in \Omega$} if $-u \in \TSHD(x_0)$ and \emph{$\Theta$-superharmonic in $\Omega$} if  $-u \in \TSHD(\Omega)$.
\item[(c)] A function $u \in C(\Omega)$ is said to be \emph{$\Theta$-harmonic in $x_0 \in \Omega$} if $u \in \TSH(x_0)$ and $-u \in \TSHD(x_0)$ and is said to be \emph{$\Theta$-harmonic in $\Omega$} if  $u \in \TSH(\Omega)$ and $-u \in \TSHD(\Omega)$.
\end{itemize}
\end{defn}
We recall that $\widetilde{\Theta}$ is elliptic if $\Theta$ is and hence the definitions (b) and (c) make sense. Moreover, since
\begin{equation}\label{bdy_Theta_char}
\partial \Theta = \Theta \cap (- \widetilde{\Theta}),
\end{equation}
the definition (c) is natural. In order to apply the above definition, the content of the following following remark will be used repeatedly.

\begin{rem}\label{use_TSH} Often we will show that $u \in \USC(\Omega)$ belongs to $\TSH(\Omega)$ by using an argument by contradiction. If $u \not\in \TSH(\Omega)$, then there exists $x_0 \in \Omega$ such that $u \not\in \TSH(x_0)$ and hence by Definition \ref{def:Weak_SH} and Lemma \ref{lemma:nonSA} there exist $v \in C^2(\Omega)$ with $D^2v(x_0) \in \widetilde{\Theta}(x_0)$ and a triple $(\veps, r, a)$ such that
$$
    \mbox{ $(u + v - a)(x_0) = 0$ \quad and \quad $(u + v -a)(x) \leq -\veps |x-x_0|^2, \ \forall \ x \in B_r(x_0)$.}
$$
Without loss of generality, by reducing $\veps$, we can assume that $v$ satisfies the stronger condition
$$
    D^2v(x_0) \in \widetilde{\Theta}^{\circ}(x_0) = -[\Theta(x_0)^c].
$$
Indeed, the perturbation $w: = v + \veps Q_{x_0}$ with
\begin{equation}\label{Qx0}
Q_{x_0}(x) = \frac{1}{2} |x-x_0|^2
\end{equation}
satisfies $D^2 w(x_0) = D^2v(x_0) + \veps I \in \widetilde{\Theta}^{\circ}(x_0)$ and
$$
\mbox{ $(u + w - a)(x_0) = 0$ \quad and \quad $(u + w -a)(x) \leq - \frac{\veps}{2} |x-x_0|^2, \ \forall \ x \in B_r(x_0)$.}
$$
\end{rem}

We now examine the relation between the notion of $\Phi$-admissible viscosity solutions (Definition \ref{Vs_def}) and $\Theta$-harmonic maps (Definition \ref{def:Weak_SH}) when the nonlinear PDE \eqref{FNE} admits an elliptic branch in accordance with Proposition \ref{pick_branch}.

\begin{prop}\label{SHCVS} Let $F \in C(\Omega \times \Ss(N), \R)$ and $\Phi: \Omega \to \mathcal{E}$ be such that \eqref{deg_ell} and \eqref{F_strict} hold and let $\Theta$ be the corresponding elliptic map defined by \eqref{def_branch1}. Then the following equivalences hold.
\begin{itemize}
\item[(a)] A function $u \in \USC(\Omega)$ is a $\Phi$-admissible viscosity subsolution of \eqref{FNE} in $\Omega$ if and only if $u \in \TSH(\Omega)$.
\item[(b)] A function $u \in \LSC(\Omega)$ is a $\Phi$-admissible viscosity supersolution of \eqref{FNE} in $\Omega$ if and only if $-u \in \TSHD(\Omega)$ provided that the branch condition \eqref{def_branch2} and the following non-degeneracy condition \footnote{This condition corresponds to part of Krylov's definition that the PDE \eqref{FNE} gives a {\em canonical form} of the branch \eqref{bdy_Theta}. See point (1) of Definition 3.1 in \cite{Kv95}.} are satisfied
    \begin{equation}\label{NDC}
    \mbox{ $ F(x,A) > 0$ for each $x \in \Omega$ and each $A \in \Theta(x)^{\circ}$.}
    \end{equation}
\end{itemize}
\end{prop}

\begin{proof} The defining condition \eqref{def_branch1} of the elliptic map $\Theta$ is given pointwise by
\begin{equation}\label{calc_dual1}
    \Theta(x) = \{ A \in \Phi(x): \ F(x,A) \geq 0 \} = \Phi \cap \{ A \in \Ss(N): \ F(x,A) \geq 0 \}, \ \ x \in \Omega
\end{equation}
and hence the dual map $\widetilde{\Theta}$ is given by
$$
    \widetilde{\Theta}(x) =  \widetilde{\Phi}(x)  \cap \overline{ \{ A \in \Ss(N): \  F(x,-A) < 0 \}},
     \ \ x \in \Omega,
$$
which follows from a simple calculation using \eqref{calc_dual1}, the definition of duality \eqref{DM} and the elementary property
$$
    \widetilde{\Phi \cap \Psi} =  \widetilde{\Phi} \cup  \widetilde{\Psi}, \ \ \Phi, \Psi \in \Ss(N).
$$
 Moreover, using the property $[\widetilde{\Theta}(x)]^{\circ} = - [ \Theta(x) ]^{c}$ and \eqref{calc_dual1} one has
 \begin{equation}\label{calc_dual4}
    [\widetilde{\Theta}(x)]^{\circ} =  \widetilde{\Phi}^{\circ} \cup  \{ A \in \Ss(N): \  F(x,-A) < 0 \}.
 \end{equation}

For part (a), assume first that $u \in \USC(\Omega)$ is a $\Phi$-admissible viscosity subsolution of \eqref{FNE} in $\Omega$ but that $u \not\in \TSH(x_0)$ for some $x_0 \in \Omega$. Exploiting Remark \ref{use_TSH}, there exists $v \in C^2(\Omega)$ with $D^2v(x_0) \in [\widetilde{\Theta}(x_0)]^{\circ}$ and there exists a triple $(\veps, r, a)$ such that
\begin{equation}\label{notTSH1}
\mbox{ $(u + v - a)(x_0) = 0$ \quad and \quad $(u + v -a)(x) \leq -\veps |x-x_0|^2, \ \forall \ x \in B_r(x_0)$.}
\end{equation}
Pick $- \varphi = v - a \in C^2(\Omega)$ and \eqref{notTSH1} says that $u-\varphi$ has a local maximum in $x_0$, while $D^2\varphi(x_0) = - D^2v(x_0) \in -[\widetilde{\Theta}(x_0)]^{\circ}$ and \eqref{calc_dual4} yields
$$
    \mbox{ $ D^2 \varphi(x_0) \in -[\widetilde{\Phi}(x_0)]^{\circ} = [\Phi(x_0)]^c$ \quad or \quad $F(x_0, D^2\varphi(x_0)) < 0$,}
$$
which contradicts $u$ being a $\Phi$-admissible viscosity subsolution.

For the converse, assume that $u \in \TSH(\Omega)$ but that $u \in \USC(\Omega)$ is not a $\Phi$-admissible viscosity subsolution of \eqref{FNE} in some $x_0 \in \Omega$. Hence there exists $\varphi \in C^2(\Omega)$ and $r > 0$ such that
$$
    (u - \varphi)(x) \leq (u - \varphi)(x_0) = 0, \ \forall x \in B_r(x_0)
$$
and
\begin{equation}\label{notVsub2}
\mbox{ $F(x_0, D^2\varphi(x_0)) < 0$ \quad or \quad $ D^2 \varphi(x_0) \not\in \Phi(x_0)$.}
\end{equation}
For each $\veps > 0$, with $Q_{x_0}$ defined by \eqref{Qx0} one has
\begin{equation}\label{notVsub3}
\mbox{ $(u - \varphi - \veps Q_{x_0})(x_0) = 0$ \quad and \quad $(u - \varphi - \veps Q_{x_0})(x) \leq - \frac{\veps}{2}|x - x_0|^2$ in $B_r(x_0)$.}
\end{equation}
Setting $v_{\veps}:= - \varphi - \veps Q_{x_0} \in C^2(\Omega)$, \eqref{notVsub3} says that $u + v_{\veps} \not\in \SA(x_0)$ for each $\veps > 0$ and hence $D^2v_{\veps}(x_0) \not\in \widetilde{\Theta}(x_0)$ for each $\veps > 0$. Using the definition of $v_{\veps}$ and the relation $[\widetilde{\Theta}(x_0)]^c = - [\Theta(x_0)]^{\circ}$ one has
\begin{equation}\label{notVsub4}
 D^2\varphi(x_0) + \veps I = - D^2v_{\veps}(x_0) \in [\Theta(x_0)]^{\circ}, \ \forall  \veps > 0.
\end{equation}
Taking the limit as $\veps \to 0^+$ in \eqref{notVsub4} (and using the continuity of $F$,  \eqref{closure_interior} and \eqref{calc_dual1}) one has
$$
    \mbox{ $D^2\varphi(x_0) \in \Theta(x_0) \subset \Phi(x_0)$ \quad and \quad $ F(x_0, D^2\varphi(x_0)) \geq 0$,}
$$
which contradicts \eqref{notVsub2}.

For part (b), assume first that $u \in \LSC(\Omega)$ is a $\Phi$-admissible viscosity supersolution of \eqref{FNE} in $\Omega$ but that $-u \not\in \TSHD(x_0)$ for some $x_0 \in \Omega$. Exploiting Remark \ref{use_TSH}, there exists $v \in C^2(\Omega)$ with $D^2v(x_0) \in [\Theta(x_0)]^{\circ}$ and there exists a triple $(\veps, r, a)$ such that
\begin{equation}\label{notTSHD1}
\mbox{ $(-u + v - a)(x_0) = 0$ \quad and \quad $(-u + v -a)(x) \leq -\veps |x-x_0|^2, \ \forall \ x \in B_r(x_0)$.}
\end{equation}
Pick $\varphi = v - a \in C^2(\Omega)$ and \eqref{notTSHD1} says that $u-\varphi$ has a local minimum in $x_0$. Since $u$ is a $\Phi$-admissible viscosity supersolution, one has
\begin{equation}\label{notTSHD2}
\mbox{ $ D^2 \varphi(x_0) \not\in \Phi(x_0)$ \quad or \quad $F(x_0, D^2\varphi(x_0)) \leq 0$.}
\end{equation}
Since $D^2\varphi(x_0) =  D^2v(x_0) \in [\Theta(x_0)]^{\circ}$ one has
\begin{equation}\label{notTSHD3}
\mbox{ $ D^2 \varphi(x_0) \in [\Phi(x_0)]^{\circ} \subset \Phi(x_0)$ \quad and \quad $F(x_0, D^2\varphi(x_0)) \geq 0$.}
\end{equation}
Combining \eqref{notTSHD2} and \eqref{notTSHD3} one must have
$$
    \mbox{ $ D^2 \varphi(x_0) \in [\Theta(x_0)]^{\circ}$ \quad and \quad $F(x_0, D^2\varphi(x_0)) = 0$,}
$$
which cannot happen if the non-degeneracy condition \eqref{NDC} holds.

For the converse, assume that $-u \in \TSHD(\Omega)$ but that $u \in \LSC(\Omega)$ is not a $\Phi$-admissible viscosity supersolution of \eqref{FNE} in some $x_0 \in \Omega$. Hence there exists $\varphi \in C^2(\Omega)$ and $r > 0$ such that
$$
    (u - \varphi)(x) \geq (u - \varphi)(x_0) = 0, \ \forall x \in B_r(x_0)
$$
and
\begin{equation}\label{notVsuper2}
\mbox{ $F(x_0, D^2\varphi(x_0)) > 0$ \quad and \quad $ D^2 \varphi(x_0) \in \Phi(x_0)$.}
\end{equation}
For each $\veps > 0$, with $Q_{x_0}$ defined by \eqref{Qx0} one has
\begin{equation}\label{notVsuper3}
\mbox{ $(-u + \varphi - \veps Q_{x_0})(x_0) = 0$ \ and \  $(-u + \varphi - \veps Q_{x_0})(x) \leq - \frac{\veps}{2}|x - x_0|^2$ in $B_r(x_0)$.}
\end{equation}
Setting $v_{\veps}:= \varphi - \veps Q_{x_0} \in C^2(\Omega)$, \eqref{notVsuper3} says that $-u + v_{\veps} \not\in \SA(x_0)$ for each $\veps > 0$ and hence $D^2v_{\veps}(x_0) \not\in \Theta(x_0)$ for each $\veps > 0$ since $-u \in \TSHD(\Omega)$. Using the definition of $v_{\veps}$ and the definition of $\Theta(x_0)$ one has
\begin{equation}\label{notVsuper4}
\mbox{ $D^2\varphi(x_0) - \veps I \not\in \Phi(x_0)$ \quad or \quad $ F(x_0, D^2\varphi(x_0) - \veps I) < 0, \ \forall  \veps > 0$.}
\end{equation}
From \eqref{notVsuper2} and the continuity of $F$ it follows that $F(x_0, D^2\varphi(x_0)- \veps I) > 0$ for each $\veps > 0$ small enough. Hence the second possibility in \eqref{notVsuper4} cannot occur for each $\veps > 0$ small and one has $D^2\varphi(x_0) - \veps I \not\in \Phi(x_0)$ for each $\veps > 0$ small. It follows that $D^2\varphi(x_0) \not\in [\Theta(x_0)]^{\circ}$ and hence \eqref{notVsuper2} yields
\begin{equation}\label{notVsuper5}
\mbox{ $F(x_0, D^2\varphi(x_0)) > 0$ \quad and \quad $ D^2 \varphi(x_0) \in \partial \Phi(x_0)$.}
\end{equation}
Using the relation \eqref{EBPhi2} and the branch condition \eqref{def_branch2} it follows that
$$
D^2 \varphi(x_0) \in \partial \Theta(x_0) \subset \{ A \in \Ss(N): \ F(x_0, A) \leq 0 \},
$$
which contradicts \eqref{notVsuper5}.

\end{proof}

We conclude this section with the following {\em coherence property} between classical and weakly $\Theta$-harmonic maps, which will be used often in the sequel.
\begin{prop}\label{coherence} Let $u \in \USC(\Omega)$ be twice differentiable in $x_0 \in \Omega$. Then
\begin{equation}\label{coherence1}
 u \in \TSH(x_0) \Leftrightarrow \frac{1}{2} \left[ D^2u(x_0) + D^2u(x_0)^T \right] \in \Theta(x_0).
\end{equation}
In particular, if $u \in C^2(\Omega)$ then
\begin{equation}\label{coherence2}
u \in \TSH(\Omega) \Leftrightarrow  D^2u(x_0) \in \Theta(x_0), \forall \ x_0 \in \Omega.
\end{equation}
\end{prop}

\begin{proof} Clearly \eqref{coherence2} follows from \eqref{coherence1}. For $u$ twice differentiable in $x_0$, considering its Taylor expansion one has the following fact: for every $\veps > 0$ there exists $r = r(\veps) > 0$ such that for each $x \in B_r(x_0)$ one has
$$
u(x) - u(x_0) - \langle Du(x_0), x - x_0 \rangle - \frac{1}{2} \langle D^2u(x_0)(x-x_0), x-x_0 \rangle - \frac{\veps}{2} |x- x_0|^2 \leq - \frac{\veps}{4} |x- x_0|^2.
$$
Setting $v_{\veps} = - \frac{1}{2} \langle D^2u(x_0)(x-x_0), x-x_0 \rangle - \frac{\veps}{2} |x- x_0|^2$ and $a = u(x_0) + \langle Du(x_0), x - x_0 \rangle$ one reads this as
$$
\mbox{$  (u + v_{\veps} - a)(x_0) = 0$ \quad and \quad $(u + v_{\veps} - a)(x) \leq - \frac{\veps}{4} |x- x_0|^2, \ \forall x \in B_r(x_0),$}
$$
which means that $u + v_{\veps} \not\in \SA(x_0)$ for each $\veps > 0$. Assuming $ u \in \TSH(x_0)$ one concludes that $D^2v_{\veps}(x_0) \not\in \widetilde{\Theta}(x_0)$ for each $\veps > 0$; that is, for each $\veps > 0$ one has
$$
- \frac{1}{2} \left[ D^2u(x_0) + D^2u(x_0)^T \right] - \veps I \in \widetilde{\Theta}(x_0)^c = - \left[ \Theta(x_0)^{\circ} \right].
$$
Letting $\veps \to 0^+$ and using \eqref{closure_interior}, one obtains $\frac{1}{2} \left[ D^2u(x_0) + D^2u(x_0)^T \right] \in \Theta(x_0)$.

For the converse, one argues by contradiction. Assume that $u$ is twice differentiable with $\frac{1}{2} \left[ D^2u(x_0) + D^2u(x_0)^T \right] \in \Theta(x_0)$ but $u \not\in \TSH(x_0)$. Then there exists $v \in C^2(\Omega)$ with $D^2v(x_0) \in \widetilde{\Theta}(x_0)$ and $(\veps, r, a)$ such that
$$
\mbox{$  (u + v- a)(x_0) = 0$ \quad and \quad $(u + v - a)(x) \leq - \veps|x- x_0|^2, \ \forall x \in B_r(x_0).$}
$$
Taking the perturbation $\hat{v}(x) = v(x) + \frac{\veps}{2} |x- x_0|^2$ one has
$$
\mbox{$  (u + \hat{v} - a)(x_0) = 0$ \quad and \quad $(u + \hat{v} - a)(x) \leq - \frac{\veps}{2} |x- x_0|^2, \ \forall x \in B_r(x_0).$}
$$
Hence $w = u + \hat{v} - a$ is twice differentiable in $x_0$ and has a local maximum in $x = x_0$. Therefore one has $D^2 w(x_0) = D^2 u(x_0) + D^2v(x_0) + \veps I \leq 0$ and hence
$$
A + B := \frac{1}{2} \left[ D^2u(x_0) + D^2u(x_0)^T \right] + D^2v(x_0) \leq - \veps I,
$$
which yields $\lambda_N(A+B) < 0$. Hence $A + B \not\in \widetilde{\mathcal{P}}$ by \eqref{nonnegative}, but this contradicts \eqref{sum_duals} since $A \in \Theta(x_0)$ and $B \in \widetilde{\Theta}(x_0)$.
\end{proof}

In light of Proposition \ref{SHCVS} and Proposition \ref{coherence}, if in addition to \eqref{deg_ell} and \eqref{F_strict} one assumes the non degeneracy condition \eqref{NDC} with respect to $\Theta$ defined by \eqref{def_branch1} then $u \in C^2(\Omega)$ satisfying \eqref{Phi_AVS} will be a $\Phi$-admissible viscosity solution of \eqref{FNE} in the sense of Definition \ref{Vs_def}. 

\section{Semicontinuity of elliptic maps and elementary properties}\label{semicontinuity}

In preparation for the implementation of a Perron method for elliptic maps and elliptic branches of \eqref{FNE}, we will present a few elementary properties which mirror well known ingredients in the classical viscosity theory. At a certain point, mild regularity properties of the set valued map $\Theta$ will play a role.

\begin{lem}\label{max_affine}
Let $\Theta$ be an elliptic map on $\Omega$. The following properties hold.
\begin{itemize}
\item[(A)] If $u \in \TSH(\Omega)$ and $a$ is affine then $u + a \in \TSH(\Omega)$.
\item[(M)] If $u,v \in \TSH(\Omega)$ then $\max\{ u, v \} \in \TSH(\Omega)$.
\end{itemize}
\end{lem}

\begin{proof}
The {\em affine property} (A) reduces to the claim that for each $x_0 \in \Omega, w \in \USC(\Omega)$ and $a$ affine one has
\begin{equation}\label{A1}
w \in \SA(x_0) \Rightarrow w + a \in \SA(x_0).
\end{equation}
Indeed, $u + a \in \TSH(\Omega)$ requires that for each $ v \in C^2(\Omega)$ with $D^2 v(x_0) \in \widetilde{\Theta}(x_0)$ one has
$$
u + v + a \in \SA(x_0),
$$
but $u + v \in \SA(x_0)$ since $u \in \TSH(x_0)$. The claim \eqref{A1} follows from a simple argument by contradiction. If $w + a \not\in \SA(x_0)$, then by Lemma \ref{lemma:nonSA} there exists a triple $(\veps, r, a^*)$ such that
$$
\mbox{$(w + a - a^*)(x_0) = 0$ \quad and \quad $(w + a - a^*)(x) \leq - \veps |x- x_0|^2, \ \forall \ x \in B_r(x_0)$,}
$$
but $a^* - a$ is affine and one concludes that $w \not\in \SA(x_0)$.

For the {\em maximum property} (M), if the conclusion were false then there exists $x_0 \in \Omega$ such that $\max\{ u, v \} + \varphi \not\in \SA(x_0)$ for some $\varphi \in C^2(\Omega)$ with $D^2 \varphi(x_0) \in \widetilde{\Theta}(x_0)$. Again by Lemma \ref{lemma:nonSA} one has a triple $(\veps, r, a)$ such that
\begin{equation}\label{M1}
 (\max\{ u, v \} + \varphi - a)(x) \leq - \veps |x- x_0|^2, \ \forall \ x \in B_r(x_0).
\end{equation}
and
\begin{equation}\label{M2}
(\max\{ u, v \} + \varphi - a)(x_0) = 0.
\end{equation}
Without loss of generality, we can write \eqref{M2} as $u(x_0) + \varphi(x_0) - a(x_0) = 0$ and \eqref{M1} as
$$
 (u + \varphi - a)(x) \leq (\max\{ u, v \} + \varphi - a)(x) \leq - \veps |x- x_0|^2, \ \forall \ x \in B_r(x_0),
$$
and hence $u + \varphi \not\in \SA(x_0)$ which contradicts $u \in \TSH(x_0)$.
\end{proof}

In order to perform various limit operations in $\TSH(\Omega)$, we will need a semicontinuity property of the set valued map $\Theta$, which mirrors what is needed for the corresponding operations for viscosity subsolutions of \eqref{FNE}. We recall the natural notion of semicontinuity for set valued maps, where given $\Phi \subset \Ss(N)$ and $\veps > 0$ we will denote by
$$
    N_{\veps} \Phi = \{ B \in \Ss(N): ||B - A|| < \veps \quad \text{for some} \ A \in \Phi \} = \bigcup_{A \in \Phi} B_{\veps}(A),
$$
the {\em $\veps$-enlargement} of the subset $\Phi$ where $\displaystyle{||A||:= \max_{1 \leq i \leq N} |\lambda_i(A)|}$ gives a norm on $\Ss(N)$.

\begin{defn}\label{defn:uscTheta} A set valued map $\Theta: \Omega \multimap \Ss(N)$ is said to be
\begin{itemize}
\item[(a)] {\em upper semicontinuous in $x_0 \in \Omega$} if
\begin{equation}\label{uscTheta}
\mbox{$\forall \veps > 0 \ \exists \delta = \delta(\veps, x_0)$ such that $\Theta(B_{\delta}(x_0)) \subset N_{\veps}(\Theta(x_0))$;}
\end{equation}
\item[(b)] {\em upper semicontinuous on $\Omega$} if this holds for every $x_0 \in \Omega$.
\end{itemize}
If, in addition, $\Theta$ takes values in the elliptic subsets $\mathcal{E}$, then $\Theta$ will be called {\em upper semicontinuous elliptic map}. The collection of all such maps will be denoted by $\USC(\Omega; \mathcal{E})$.
\end{defn}

Since elliptic sets are closed, if $\Theta(x_0) \in \mathcal{E}$ then \eqref{uscTheta} implies the following statement 
\begin{equation}\label{uscTheta_n}
\left\{ \begin{array}{l} x_n \to x_0 \ \text{in} \ \Omega \\ A_n \in \Theta(x_n) \\ A_n \to A_0 \ \text{in} \ \Ss(N) \end{array} \right. \Rightarrow A_0 \in \Theta(x_0),
\end{equation}
which will be used in the following lemma on limit operations.

\begin{lem}\label{limit_sup} Let $\Theta \in \USC(\Omega; \mathcal{E})$. Then the following properties hold.
\begin{itemize}
\item[(L)] If $\{u_n\}_{n \in \N} \subset \TSH(\Omega)$ is a decreasing sequence, then the limit $\displaystyle{u:= \lim_{n \to +\infty} u_n}$ belongs to $\TSH(\Omega)$.
\item[(S)] If $\mathcal{F} \subset \TSH(\Omega)$ is a non empty family of functions which are locally uniformly bounded from above, then the Perron function $\displaystyle{u:= \sup_{f \in \mathcal{F}} f}$ has $u^* \in \TSH(\Omega)$, where
$$
    u^*(x):= \limsup_{r \to 0^+} \{ u(y): y \in \Omega \cap \overline{B}_r(x_0)\}, x \in \Omega
$$
is the upper semicontinuous regularization of $u$.
\end{itemize}
\end{lem}

\begin{proof}
The proof of the {\em decreasing limit property} (L) will use the following consequence of Cantor's intersection theorem applied to the decreasing sequence $\{u_n\}_{n \in \N} \subset \TSH(\Omega)$: {\em for each compact subset $K$ of $\Omega$ the upper semicontinuous function $\displaystyle{u:= \lim_{n \to +\infty} u_n}$ satisfies}
\begin{equation}\label{L1}
\sup_K u = \lim_{n \to +\infty} \left( \sup_K u_n \right).
\end{equation}
For a proof of this fact, see Appendix B of \cite{HL11}. To prove (L), one argues by contradiction assuming that $u \not\in \TSH(x_0)$ for some $x_0 \in \Omega$. Again using Remark \ref{use_TSH}, there exists $v \in C^2(\Omega)$ such that $D^2 v(x_0) \in [\widetilde{\Theta}(x_0)]^{\circ}$ and there exists a triple $(\veps, R, a)$ such that
\begin{equation}\label{L2}
\mbox{ $(u + v - a)(x_0) = 0$ \quad and \quad $(u + v - a)(x) \leq - \veps |x - x_0|^2, \forall \ x \in B_R(x_0)$.}
\end{equation}
Using the relation $ - \widetilde{\Theta}(x_0)^{\circ} = [\Theta(x_0)]^c$ one has
\begin{equation}\label{L3}
- D^2 v(x_0) \not\in \Theta(x_0).
\end{equation}
The idea now is to construct a sequence $x_n \to x_0$ for which
\begin{equation}\label{L3.5}
- D^2 v(x_n) \in \Theta(x_n).
\end{equation}
If this can be done, then since $v$ is $C^2$ and $\Theta$ is upper semicontinuous (see \eqref{uscTheta_n}), one has $- D^2 v(x_0) \in \Theta(x_0)$, which contradicts \eqref{L3}.

We proceed to construct the desired sequence. With $Q_{x_0}$ defined by \eqref{Qx0}, set $\varphi \in C^2(\Omega)$ by $- \varphi := v -a  + \veps Q_{x_0}$ . Using \eqref{L2}, for each $x \in B_R(x_0)$ one has
\begin{equation}\label{L4}
(u - \varphi)(x) \leq - \frac{ \veps}{2} |x - x_0|^2 \leq 0 = (u - \varphi)(x_0)
\end{equation}
and hence $u - \varphi$ has a strict maximum value of zero on each compact ball $\overline{B}_r(x_0) \subset B_R(x_0)$ with $r < R$. Since $u_n - \varphi$ is upper semicontinuous on the compact set $\overline{B}_r(x_0)$, there exists $x_n \in \overline{B}_r(x_0)$ such that
$$
(u_n - \varphi)(x_n) = \max_{\overline{B}_r(x_0)} (u_n - \varphi).
$$
For each $\rho \in (0,r)$ consider the compact set $K:= \overline{B}_r(x_0) \setminus {B}_{\rho}(x_0)$ and apply \eqref{L1} to the decreasing sequence $\{u_n - \varphi \}_{n \in \N} \subset \USC(\Omega)$ to find
$$
\lim_{n \to +\infty} \left( \sup_K [u_n - \varphi] \right) = \sup_K (u - \varphi) < 0,
$$
where we have also used \eqref{L4}. Hence one has
\begin{equation}\label{L5}
\sup_K (u_n - \varphi) < 0 \quad \text{for each $n$ large}.
\end{equation}
However, since $\{u_n\}$ is decreasing \eqref{L4} also yields
$$
    \max_{\overline{B}_r(x_0)}(u_n - \varphi) \geq \max_{\overline{B}_r(x_0)}(u - \varphi) = 0 = (u - \varphi)(x_0)
$$
and hence  \eqref{L5} implies that $x_n \not\in K$ for each $n$ large. Thus $x_n \in B_{\rho}(x_0)$ with $\rho < r$ arbitrary and hence $x_n \to x_0$ as $n \to +\infty$. In particular, $x_n \in B_{\rho}(x_0)$ is an interior maximum point for $u_n - \varphi$.

We claim
$$
     D^2 v(x_n) \not\in \widetilde{\Theta}(x_n) = [- \Theta(x_n)^{\circ} ]^c
$$
and hence $- D^2 v(x_n) \in \Theta(x_n)^{\circ}$ which is stronger than the condition \eqref{L3.5} needed for the contradiction. Using $u_n \in \TSH(x_n)$ it suffices to find triples $(\veps_n, r_n, a_n)$ such that
\begin{equation}\label{L8}
\mbox{ $(u_n + v - a_n)(x_n) = 0$ \quad and \quad $(u_n + v - a_n)(x) \leq -  \veps_n |x - x_n|^2, \forall \ x \in B_{r_n}(x_n)$,}
\end{equation}
where $x_n$ realizes the maximum on $\overline{B}_r(x_0)$ of $u_n - \varphi = u_n + v - a + \veps Q_{x_0}$. Choose $r_n$ small enough to ensure that $\overline{B}_{r_n}(x_n) \subset \overline{B}_r(x_0)$ and $x_n$ realizes the maximum on $\overline{B}_{r_n}(x_n)$.
Hence one has
$$
(u_n + v - a)(x) + \frac{\veps}{2}|x-x_0|^2 \leq M_n := (u_n + v - a)(x_n) + \frac{\veps}{2}|x_n-x_0|^2, \ \forall \ x \in B_{r_n}(x_n),
$$
which is equivalent to
$$
    (u_n + v - a + \veps Q_{x_0} - \veps Q_{x_n})(x) - M_n \leq - \frac{\veps}{2}|x-x_n|^2, \ \forall \ x \in B_{r_n}(x_n),
$$
where $(u_n + v - a + \veps Q_{x_0} - \veps Q_{x_n})(x_n) - M_n = 0$. This is \eqref{L8} with $\veps_n = \veps/2$ and
$$
a_n(x) := a(x) - \frac{\veps}{2}|x-x_0|^2 + \frac{\veps}{2}|x-x_n|^2 + M_n,
$$
which is affine.

For the {\em supremum over locally bounded families property} (S), the argument is also by contradiction. Suppose that $u^* \not\in \TSH(x_0)$ for some $x_0 \in \Omega$. Hence the exist $v \in C^2(\Omega)$ with $D^2 v(x_0) \in \widetilde{\Theta}(x_0)^{\circ} = - [\Theta(x_0)^c]$ and a triple $(\veps, r, a)$ such that
\begin{equation}\label{S1}
\mbox{ $(u^* + v - a)(x_0) = 0$ \quad and \quad $ (u^* + v - a)(x) \leq  - \veps|x-x_0|^2, \ \forall \ x \in B_{r}(x_0),$}
\end{equation}
where we have again used Remark \ref{use_TSH} to select $D^2 v(x_0)$ in the interior of $\widetilde{\Theta}(x_0)$. Using
$$
u^*(x) = \lim_{k \to + \infty} \left( \sup_{y \in \overline{B}_{1/k}(x_0)} \left\{ \sup_{f \in \mathcal{F}} f(y) \right\} \right)
$$
there exist sequences $\{y_k\}_{k \in \N} \subset \Omega$ and $\{f_k\}_{k \in \N} \subset \mathcal{F}$ such that
\begin{equation}\label{S2}
y_k \to x_0 \quad \text{and} \quad f_k(y_k) \to u^*(x_0).
\end{equation}
As in the proof of (L), the idea is to construct a sequence $x_k \to x_0$ such that
\begin{equation}\label{S3}
- D^2 v(x_k) \in \Theta(x_k)^{\circ} = -[\widetilde{\Theta}(x_k)]^{c},
\end{equation}
which leads to a contradiction to $- D^2 v(x_0) \not\in \Theta(x_0)$ by the upper semicontinuity of $\Theta$ and the regularity of $v$.

To construct the desired sequence, define $- \varphi :=  v -a + \veps Q_{x_0}$ and use \eqref{S1} to find
\begin{equation}\label{S4}
\mbox{$(u^* - \varphi)(x_0) = 0$ \quad and \quad $ (u^* - \varphi)(x) \leq  - \frac{\veps}{2} |x-x_0|^2, \ \forall \ x \in B_{r}(x_0).$}
\end{equation}
For each $k \in \N$, select $x_k \in \overline{B}_{\rho}(x_0) \subset B_r(x_0)$ to be a point which realizes the maximum over $\overline{B}_{\rho}(x_0)$ of the upper semicontinuous function $f_k - \varphi$. Extract a subsequence (still called $\{x_k\}_{k \in \N}$) such that $x_k \to \hat{x}$ for some $\hat{x} \in  \overline{B}_{\rho}(x_0)$. We claim $\hat{x} = x_0$. Indeed, since $y_k \to x_0$ for each $k$ large enough one has
$$
f_k(y_k) - \varphi(y_k) \leq f_k(x_k) - \varphi(x_k).
$$
Letting $k \to +\infty$ in this inequality and using \eqref{S2}, \eqref{S4}, the definition of $u^*$ and the continuity of $\varphi$ one has
\begin{align*}
0 & = \liminf_{k \to +\infty} [f_k(y_k) - \varphi(y_k)] \leq \liminf_{k \to +\infty} f_k(x_k) - \varphi(\hat{x}) \\
    & = \limsup_{k \to +\infty} u^*(x_k) - \varphi(\hat{x}) \leq u^*(\hat{x}) - \varphi(\hat{x}),
\end{align*}
where by \eqref{S1} and the definition of $\varphi$ one has
$$
u^*(x) - \varphi(x) \leq - \frac{\veps}{2} |x - x_0|^2 \leq 0 = (u^* - \varphi)(x_0) , \ \forall \ x \in B_{r}(x_0).
$$
Hence $\hat{x} = x_0$ as claimed. In particular, $x_k \in B_{\rho}(x_0)$ for each large $k$.

It remains only to verify that \eqref{S3} holds. Since $f_k - \varphi = f_k + v -a + \veps Q_{x_0}$ takes on its maximum value $M_k$ over $\overline{B}_{\rho}(x_0) \subset B_r(x_0)$ in the point $x_k$, one has
\begin{equation}\label{S5}
\mbox{$(f_k + v - a_k)(x_k) = 0$ \quad and \quad $ (f_k + v - a_k)(x_k) \leq  - \frac{\veps}{2} |x-x_k|^2, \ \forall \ x \in B_{\rho}(x_0),$}
\end{equation}
where $a_k:= a + \veps Q_{x_0} - \veps Q_{x_k} - M_k$ is affine. From \eqref{S5} it follows that $D^2v(x_k) \not\in \widetilde{\Theta}(x_k)$ since $f_k \in \TSH(x_k)$.
\end{proof}

\section{Uniform upper semicontinuity and the comparison principle}\label{uusc_cp}

In this section, we will give a simple sufficient condition on an elliptic map $\Theta$ which ensures the validity of the {\em comparison principle}; that is, if $u \in \USC(\overline{\Omega})$ and $w \in \LSC(\overline{\Omega})$ are $\Theta$-subharmonic and $\Theta$-superharmonic respectively in $\Omega$, then
\begin{equation}\label{CP_formula}
\mbox{$u \leq w$ on $\partial \Omega \ \ \Rightarrow \ \ u \leq w$ in $\Omega$.}
\end{equation}
If at least one of the functions were regular, say $w \in C^2(\Omega)$, then by setting $v:=-w$ the coherence property Proposition \ref{coherence} yields $D^2v(x) \in \widetilde{\Theta}(x)$ for every $x \in \Omega$. Since $u \in \TSH(\Omega)$ one has then $u+v \in \SA(\Omega)$; that is, for each $a$ affine
$$
\mbox{$u + v \leq a$ on $\partial \Omega \ \ \Rightarrow \ \ u + v \leq a$ in $\Omega$,}
$$
which is just \eqref{CP_formula} for $a = 0$. Hence, the point is to show that \eqref{CP_formula} holds if both $u$ and $w$ are just semi-continuous. As shown in \cite{HL09}, by using the {\em maximum principle for subaffine functions} one can reduce the the comparison principle for semi-continuous $u$ and $w$ to the validity the so-called {\em subaffine theorem}. When $\Theta$ is a constant elliptic map, the subaffine theorem is the content of Theorem 6.5 of \cite{HL09} and we will show that it continues to hold provided that the elliptic map is {\em uniformly upper semicontinuous.}

We begin with the needed notion of regularity, which is just the uniform version of Definition \ref{defn:uscTheta}.
\begin{defn}\label{defn:uusc} A set valued map $\Theta: \Omega \multimap \Ss(N)$ is said to be {\em uniformly upper semicontinuous in $\Omega$} if for every $\veps > 0$ there exists $\delta = \delta(\veps)$ such that
\begin{equation}\label{uusc1}
\Theta(B_{\delta}(x_0) \cap \Omega) \subset N_{\veps}(\Theta(x_0)), \ \ \forall \ x_0 \in \Omega.
\end{equation}
\end{defn}
When $\Theta$ is an elliptic map (i.e.\ $\Theta$ takes values in $\mathcal{E}$), we have the following equivalent formulation of uniform upper semicontinuity, which will be used repeatedly.

\begin{prop}\label{uusc} An elliptic map $\Theta$ is uniformly upper semicontinuous in $\Omega$ if and only if for every $\veps > 0$ there exists $\delta = \delta(\veps)$ such that
\begin{equation}\label{uusc2}
\mbox{$x,y \in \Omega$ with $|x-y| < \delta \ \Rightarrow \ \Theta(x) + \veps I \subset \Theta(y)$ \ \ and \ \ $\Theta(y) + \veps I \subset \Theta(x)$.}
\end{equation}
\end{prop}

\begin{proof} For $x,y \in \Omega$ with $|x-y| < \delta$ and using \eqref{uusc1} with $x_0 = y$ and $x_0 = x$ one has
\begin{equation}\label{uusc3}
\mbox{$ \Theta(x) \subset N_{\veps}(\Theta(y))$ \quad and \quad $ \Theta(y) \subset N_{\veps}(\Theta(x))$.}
\end{equation}
For each $A \in \Theta(x)$, the first inclusion in \eqref{uusc3} yields $A = B + M$ with $B \in \Theta(y)$ and $||M|| < \veps$ so that
$$
    A + \veps I = B + M + \veps I \in \Theta(y) + (M + \veps I) \in \Theta(y)
$$
since $M + \veps I \in \mathcal{P}$ and $\Theta(y)$ is an elliptic set. Thus the first inclusion in \eqref{uusc2} holds. A similar argument gives the second inclusion in \eqref{uusc2}.

Conversely, assuming \eqref{uusc2} and setting $x_0 = y$ one has
$$
    \Theta(x) + \veps I \in \Theta(x_0), \ \forall \ x \in B _{\delta}(x_0) \cap \Omega.
$$
Hence each $A \in \Theta(x)$ can be written as $(A + \veps I) - \veps I := B + M$ with $B \in \Theta(x_0)$ and $||M|| = \veps < 2 \veps$, which yields \eqref{uusc1} with $2 \veps$ in place of $\veps$, for example.
\end{proof}

Uniform upper semicontinuity is preserved when passing to the dual map.

\begin{prop}\label{uuscd} An elliptic map $\Theta$ is uniformly upper semicontinuous in $\Omega$ if and only the dual map $\widetilde{\Theta}$ is.
\end{prop}

\begin{proof} Assuming that $\Theta$ is uniformly upper semicontinuous, let $\veps > 0$ and $\delta > 0$ be as in Definition \ref{defn:uusc}. For $x,y \in \Omega$ with $|x-y|< \delta$, formula \eqref{uusc2} yields $\Theta(y) + \veps I \subset \Theta(x)$, so $\widetilde{\Theta}(x) \subset \widetilde{\Theta(y) + \veps I} = \widetilde{\Theta}(y) - \veps I$ by the elementary properties
$$
    \Theta_1 \subset \Theta_2 \ \Rightarrow \widetilde{\Theta_2} \subset \widetilde{\Theta_1}, \ \ \Theta_1, \Theta_2 \in \mathcal{E}
$$
and
$$
    \widetilde{\Theta + A} = \widetilde{\Theta} - A, \ \ \Theta \in \mathcal{E}, A \in \Ss(N)
$$
of elliptic duals. These properties follow directly from the definition \eqref{DM} (as shown in Section 4 of \cite{HL09}). Hence $\widetilde{\Theta}(x) + \veps I \subset \widetilde{\Theta}(y)$ and therefore the elliptic map $\widetilde{\Theta}$ is uniformly upper continuous Proposition \ref{uusc}. The converse uses the same argument.
\end{proof}

Elliptic maps take values in $\mathcal{E} \subset \mathcal{K}(\Ss(N))$, where $\mathcal{K}(\Ss(N))$ are the closed subsets of $\Ss(N)$. If one considers the {\em Hausdorff distance} on $\mathcal{K}(\Ss(N))$ defined by
\begin{equation}\label{HD}
\mbox{ $d_{\mathcal{H}}(\Phi, \Psi) := \inf \{ r > 0: \ \Phi \subset N_r(\Psi) \ \text{and} \ \Psi \subset N_r(\Phi) \}$,}
\end{equation}
then one knows that $(\mathcal{K}(\Ss(N)), d_{\mathcal{H}})$ is a complete metric space since $\Ss(N)$ is a complete with respect to the metric ${\rm dist}(A,B) = ||A-B||$ (see Proposition 7.3.3 and Proposition 7.3.7 of \cite{BBI01}, for example). Since the subsets of $\Ss(N)$ need not be bounded, the metric can take on the value $+ \infty$; in particular, one has
\begin{equation}\label{HDempty}
\mbox{ $d_{\mathcal{H}}(\Phi, \emptyset) = +\infty$ for each non empty $\Phi \in \mathcal{K}(\Ss(N))$.}
\end{equation}
A surprising fact is that the uniform upper semicontinuity of an elliptic map $\Theta$ is equivalent to the uniform continuity of $\Theta$ with respect to the metric topology on $\mathcal{K}(\Ss(N))$.

\begin{prop}\label{UCHD} A set valued map $\Theta: \Omega \multimap \Ss(N)$ which takes values in the elliptic subsets $\mathcal{E}$ is uniformly upper semicontinuous in $\Omega$ if and only if $\Theta : \Omega \to \mathcal{E} \subset \mathcal{K}(\Ss(N))$ is uniformly continuous; that is, if for each $\veps > 0$ there exists $\delta = \delta(\veps)$ such that
\begin{equation}\label{UC}
\mbox{ $d_{\mathcal{H}}(\Theta(x), \Theta(y)) < \veps$ for each $x,y \in \Omega$ such that $|x-y|< \delta$.}
\end{equation}
\end{prop}

\begin{proof} We will make use of the following equivalent representation of the Hausdorff distance \eqref{HD}
\begin{equation}\label{HD2}
\mbox{ $\displaystyle{ d_{\mathcal{H}}(\Phi, \Psi) = \max \{ \sup_{A \in \Phi} \inf_{B \in \Psi} ||A-B||, \sup_{B \in \Psi} \inf_{A \in \Phi} ||A-B|| \}}$.}
\end{equation}
Assuming that $\Theta$ is uniformly upper semicontinuous, let $\veps > 0$ and $\delta = \delta(\veps)$ be as in the definition \eqref{uusc2} and let $x,y \in \Omega$ be such that $|x - y| < \delta$. For each $A \in \Theta(x)$ there exits $B \in \Theta(y)$ such that $A + \frac{\veps}{2} I = B$ and hence
$$
        \mbox{ $\displaystyle{\inf_{B \in \Theta(y)} ||A-B|| < \veps}$ for each $A \in \Theta(x)$ and each $y \in B_{\delta}(x) \cap \Omega$,}
$$
which yields
\begin{equation}\label{UC2}
     \mbox{ $\displaystyle{ \sup_{A \in \Theta(x)} \inf_{B \in \Theta(y)} ||A-B|| < \veps}$ for each $x,y \in \Omega$ such that $|x-y|< \delta$.}
\end{equation}
Starting from the second inclusion in \eqref{uusc2}, one obtains
\begin{equation}\label{UC3}
     \mbox{ $\displaystyle{ \sup_{B \in \Theta(y)} \inf_{A \in \Theta(x)} ||A-B|| < \veps} $ for each $x,y \in \Omega$ such that $|x-y|< \delta$.}
\end{equation}
Combining \eqref{UC2} and \eqref{UC3} and using \eqref{HD2} yields the uniform continuity \eqref{UC}.

Conversely, if $\Theta$ is uniformly continuous, let $\veps > 0$ and $\delta = \delta(\veps)$ be as in the definition \eqref{UC}. Using the first term in the representation \eqref{HD2}, for each $x,y \in \Omega$ with $|x-y|< \delta$ one has
$$
     \mbox{ $\displaystyle{ \inf_{B \in \Theta(y)} ||A-B|| < \veps}$ for each $A \in \Theta(x)$,}
$$
and hence there exists $B \in \Theta(y)$ such that $A - B:= M$ satisfies $||M|| \leq \veps$. Since $M + \veps I \in \mathcal{P}$ and $\Theta(y)$ is elliptic, one has
$$
A + \veps I = B + (M + \veps I) \in \Theta(y);
$$
that is,
$$
     \mbox{ $\Theta(x) + \veps I \subset \Theta(y)$ for each $x,y \in \Omega$ such that $|x-y|< \delta$,}
$$
which is the first inclusion in \eqref{uusc2}. Interchanging the roles of $x$ and $y$ and using the second term in \eqref{HD2} gives the second inclusion in \eqref{uusc2}.
\end{proof}

Exploiting this equivalence, one can prove that uniformly upper semicontinuous elliptic maps on bounded domains can be extended to the boundary, where the uniform upper semicontinuity on $\overline{\Omega}$ just means that \eqref{uusc2} holds for all $x,y \in \overline{\Omega}$ with $|x-y| < \delta$.

\begin{prop}\label{extension} Let $\Theta$ be a uniformly upper semicontinuous elliptic map on $\Omega$. Then $\Theta$ extends to a a uniformly upper semicontinuous elliptic map on $\overline{\Omega}$.
\end{prop}

\begin{proof} We first construct the extension in the obvious way. Consider an arbitrary $x_0 \in \partial \Omega$ and select a sequence $\{x_k\}_{k \in \N} \subset \Omega$ such that $x_k \to x_0$ as $k \to +\infty$. Since $\{x_k\}_{k \in \N}$ is a Cauchy sequence and $\Theta$ is uniformly continuous by Proposition \ref{UCHD}, $\{\Theta(x_k)\}_{k \in \N}$ is a Cauchy sequence in $\mathcal{K}(\Ss(N))$, which is complete and hence there exists $\Theta(x_0) \in \mathcal{K}(\Ss(N))$ such that
\begin{equation}\label{HDL}
    d_{\mathcal{H}}(\Theta(x_k), \Theta(x_0)) \to 0 \ \text{as} \ k \to +\infty.
\end{equation}
This limiting set clearly does not depend on the sequence $\{x_k\}_{k \in \N}$ chosen. Doing this for each $x_0 \in \partial \Omega$ extends $\Theta$ to a uniformly continuous map from $\overline{\Omega}$ taking values in $\mathcal{K}(\Ss(N))$.

By construction, each limiting set $\Theta(x_0)$ is closed and must be nonempty, since otherwise \eqref{HDempty} would contradict the convergence \eqref{HDL}. It remains only to show that each $\Theta(x_0)$ is proper and satisfies the positivity property \eqref{elliptic2}. For the positivity property, one uses the fact each $A \in \Theta(x_0)$ is a limit in $\Ss(N)$ of a sequence $\{A_k\}_{k \in \N}$ with $A_k \in \Theta(x_k)$. Hence for each $A \in \Theta(x_0)$ and each $P \in \mathcal{P}$ one has
$$
   \mbox{$\displaystyle{ A + P = \lim_{k \to +\infty} (A_k + P)}$ \quad with \quad $A_k + P \in \Theta(x_k)$,}
$$
and hence $\Theta(x_0) + \mathcal{P} \subset \Theta(x_0)$, as desired. Finally, to show that each $\Theta(x_0)$ is proper, it suffices to show that
\begin{equation}\label{Ext2}
   \mbox{ $d_{\mathcal{H}}(\Theta, \Ss(N)) = +\infty$ for each elliptic set $\Theta$.}
\end{equation}
Indeed, if $\Theta_0 = \Ss(N)$ then applying \eqref{Ext2} with $\Theta = \Theta(x_k)$ would contradict the convergence \eqref{HDL}. To show that \eqref{Ext2} holds, it suffices to show that $\Theta^c = \Ss(N) \setminus \Theta$ contains balls of arbitrarily large radius so that no finite enlargement of $\Theta$ can exhaust $\Ss(N)$. For each elliptic $\Theta$, one has
$$
    \Theta^c = - \left( \widetilde{\Theta}^{\circ} \right)
$$
and hence $\Theta^c$ contains an open ball about some element $A_0 \in - \left( \widetilde{\Theta}^{\circ} \right)$. Translation by a fixed element of $\Ss(N)$ preserves the ellipticity of  $\Theta$ and hence one may assume that $A_0 = 0$. By the ellipticity of $\widetilde{\Theta}$ one has
$$
    - \widetilde{\Theta} - \mathcal{P} \subset - \widetilde{\Theta},
$$
and since $0 \in - \widetilde{\Theta}$ one concludes that $- \mathcal{P} \subset - \widetilde{\Theta}$ and hence
$$
     \mbox{ $(-\mathcal{P})^{\circ} \subset  \left(- \widetilde{\Theta} \right)^{\circ} = \Theta^c.$}
$$
It is easy to see that $(-\mathcal{P})^{\circ} \subset \Ss(N)$ contains balls of arbitrarily large radius. For example, for each $t < 0$ one has $tI \in (-\mathcal{P})^{\circ}$
and
$$
    \{ A \in \Ss(N): \ ||tI - A|| < |t| \} \subset (-\mathcal{P})^{\circ}.
$$
Indeed, for each $t < 0$ one has
$$
|t| > || \, |t| + A \, || = \max_{1 \leq k \leq N} \left| \lambda_k( |t|I + A) \right| =  \max_{1 \leq k \leq N} \left| |t| + \lambda_k(A) \right|,
$$
and hence $\lambda_k(A) < 0$ for each $k$. Hence $N_{|t|}(tI) \subset (-\mathcal{P})^{\circ} \subset \Theta^c$.
\end{proof}

This extension result will be useful for the applications to elliptic branches of \eqref{FNE}. More precisely, we will often require control on the associated elliptic map up to the boundary, but we would prefer to impose any needed structural conditions on $F(x,A)$ only for $x \in \Omega$. See Proposition \ref{UCbranch} for one such illustration.

As a final preparatory ingredient, the following consequence of the uniform upper semicontinuity of $\Theta$ will play a key role in the proof of the subaffine theorem, on which the comparison principle is based. It will also be used in the proof of interior continuity of the solution to the Dirichlet problem for $\Theta$-harmonic functions (see Step 5 of the proof of Theorem \ref{thm:EU} below) \footnote{Note that if $\Theta$ is a constant elliptic map, then a stronger consequence than \eqref{WTP_formula} follows, namely $u_{y; \veps} \in \TSH(\Omega_{\delta})$ for all $y \in B_{\delta}(0)$. This fact is a key ingredient in \cite{HL09}, and it may fail if $\Theta$ is a non-constant map. Uniform upper semicontinuity of $\Theta$ is used here to guarantee the milder property \eqref{WTP_formula}, which is sufficient for our purposes.}.

\begin{prop}\label{WTP} If $\Theta$ is a uniformly upper semicontinuous elliptic map on $\Omega$ and $u \in \TSH(\Omega)$, then the following property holds:
\begin{equation}\label{WTP_formula}
 \mbox{ $\forall \ \veps > 0 \ \exists \, \delta = \delta(\veps)$ such that $u_{y; \veps} := u( \cdot + y) + \frac{\veps}{2} | \cdot |^2 \in \TSH(\Omega_{\delta}), \ \forall \ y \in B_{\delta}(0)$,}
\end{equation}
where
\begin{equation}\label{Omega_delta}
\Omega_{\delta} := \{ x \in \Omega: \ {\rm dist}(x, \partial \Omega) > \delta \}.
\end{equation}
\end{prop}

\begin{proof} With $\veps, \delta$ as in Definition \ref{defn:uusc} one needs to show that for each $x_0 \in \Omega_{\delta}$ and $y \in B_{\delta}(0)$ fixed one has
\begin{equation}\label{wtp1}
\mbox{ $u_{y: \veps} + v \in \SA(x_0), \ \ \forall v \in C^2(\Omega)$ with $D^2v(x_0) \in \widetilde{\Theta}(x_0)$.}
\end{equation}
Defining the test function $\hat{v}_{y; \veps}$ by $\hat{v}_{y; \veps}(x) = v(x-y) + \frac{\veps}{2}|x-y|^2$ one has that
\begin{equation}\label{wtp2}
D^2 \hat{v}_{y; \veps}(x_0 + y) = D^2 v(x_0) + \veps I \in \widetilde{\Theta}(x_0) + \veps I \subset \widetilde{\Theta}(x_0 + y)
\end{equation}
by the uniform upper semicontinuity of $\widetilde{\Theta}$ since $x_0, x_0 + y \in \Omega$ with $| (x_0 + y) - x_0| < \delta$. Since $u \in \TSH(x_0 + y)$ and $\hat{v}_{y; \veps} \in C^2(\Omega)$ satisfies \eqref{wtp2}, one has
$$
    u + \hat{v}_{y; \veps}  \in \SA(x_0 + y)
$$
and hence
\begin{equation}\label{wtp3}
u( \cdot + y) + \hat{v}_{y; \veps}(\cdot + y)  \in \SA(x_0),
\end{equation}
since subaffinity is preserved by translations. The affirmation \eqref{wtp3} is precisely the needed relation \eqref{wtp1} by how $u_{y: \veps}$ and $\hat{v}_{y; \veps}$ are defined.
\end{proof}

\subsection{The subaffine theorem and the comparison principle}

Uniform upper semicontinuity of an elliptic map is a sufficient condition for the validity of the subaffine theorem, which will be proven in the following subsection.

\begin{thm}\label{thm:SAT} Let $\Theta$ be a uniformly upper semicontinuous elliptic map on $\Omega$. For each pair $u \in \TSH(\Omega)$ and $v \in \TSHD(\Omega)$, one has $u + v \in \SA(\Omega)$.
\end{thm}

The subaffine theorem combined with the following weak maximum principle for subaffine functions yields the comparison principle for uniformly upper semicontinuous elliptic maps.

\begin{lem}\label{MPSA} Let $K \subset \R^N$ be compact and $u \in \USC(K)$. If $u \in \SA(K^{\circ})$ then
$$
    \sup_K u \leq \sup_{\partial K} u.
$$
\end{lem}
For a proof of Lemma \ref{MPSA} see Proposition 2.3 of \cite{HL09}. We mention only that the idea is to exhaust $K^{\circ}$ by compact sets on which the maximum principle holds.

The desired comparison principle is a direct corollary of Theorem \ref{thm:SAT} and Lemma \ref{MPSA}.

\begin{thm}\label{thm:CP} Let $\Theta$ be a uniformly upper semicontinuous elliptic map on $\Omega$. Then the comparison principle holds; that is,
if $u \in \USC(\overline{\Omega})$ and $w \in \LSC(\overline{\Omega})$ are $\Theta$-subharmonic and $\Theta$-superharmonic respectively in $\Omega$, then
$$
\mbox{$u \leq w$ on $\partial \Omega \ \ \Rightarrow \ \ u \leq w$ in $\Omega$.}
$$
\end{thm}

\begin{proof}
By setting $v := -w$ one has $v \in \USC(\overline{\Omega})$ and $v \in \TSHD(\Omega)$. Since $\Theta$ is uniformly upper semicontinuous, Theorem \ref{thm:SAT} yields $u + v \in \SA(\Omega)$ and hence Lemma \ref{MPSA} gives the desired result.
\end{proof}

We remark that the comparison principle has been stated for (abstract) uniformly upper semicontinuous maps, which when combined with the equivalence of Proposition \ref{SHCVS} yields a comparison principle for $\Phi$-admissible viscosity solutions of the nonlinear PDE \eqref{FNE} when $F$ admits a uniformly upper semicontinuous branch function $\Theta$ for which the non-degeneracy condition \eqref{NDC} holds. 

\subsection{Proof of the subaffine theorem}

We now justify the subaffine theorem for uniformly upper semicontinuous elliptic maps, which we will divide into three steps. We first recall that if $\lambda > 0$, a function $u : \Omega \to \R$ is said to be {\em $\lambda$-semi-convex} if $u + \lambda Q_0$ is a convex function, where $Q_0(x) = \frac{1}{2}|x|^2$. Since elliptic subsets are subsets of the symmetric matrices $\Ss(N)$, in all that follows, the Hessian $D^2 u(x_0)$ of a twice differentiable $u$ should be replaced by $\frac{1}{2} [D^2 u(x_0) + D^2 u(x_0)^{T}]$ if the Hessian fails to be symmetric.

\vspace{2ex}

\noindent {\bf Step 1:} {\em The subaffine theorem holds for $\lambda$-semi-convex functions; that is, if $u \in \TSH(\Omega)$ and $v \in \TSHD(\Omega)$ are $\lambda$-semi-convex then $u + v \in \SA(\Omega)$.}

By Alexandroff's theorem, $u$ and $v$ are twice differentiable a.e.\ with respect to Lebesgue measure on $\Omega$ and by the coherence property Proposition \ref{coherence} one has
$$
    \mbox{$D^2 u(x_0) \in \Theta(x_0)$ \quad and \quad $D^2 v(x_0) \in \widetilde{\Theta}(x_0)$ \ \ for a.e.\ $x_0 \in \Omega$,}
$$
Making use of the property \eqref{sum_duals} and the definition of $\lambda$-semi-convexity one has
\begin{equation}\label{sat2}
\mbox{$u + v$ is $2 \lambda$-semi-convex \quad and \quad $u + v \in \widetilde{\mathcal{P}}(x_0)$ \ \ for a.e.\ $x_0 \in \Omega$.}
\end{equation}
By Proposition 4.5 of \cite{HL09}, one has the following fact: if $u \in \USC(\Omega)$ then
\begin{equation}\label{sat3}
    u \in \SA(\Omega) \ \Leftrightarrow \ u \in \PSHD(\Omega).
\end{equation}
Combining \eqref{sat2} and \eqref{sat3} the desired conclusion follows from the following key fact applied to $w = u + v$.

\begin{lem}\label{SA_ae}
Let $w \in \USC(\Omega)$ be $2\lambda$-semi-convex on $\Omega$. If $D^2 w(x) \in \widetilde{\mathcal{P}}$ for a.e.\ $x \in \Omega$, then $w \in \PSHD(\Omega)$
\end{lem}

This Lemma is Theorem 7.3 of \cite{HL09}. Given its crucial role, we will reproduce the proof below after a few preparatory remarks. Recalling that the constant elliptic map  $\widetilde{\mathcal{P}}$ is characterized by \eqref{nonnegative}, one needs to pass a lower bound on the largest eigenvalue from a subset of full measure to the entire domain $\Omega$. This is accomplished by using Slodkowski's largest eigenvalue theorem (see Corollary 3.5 of \cite{Sl84}): {\em if $v$ is a convex function on $\Omega$ and $\Lambda \geq 0$ then} 
\begin{equation}\label{slod1}
\mbox{$K(v,x) \geq \Lambda$ for a.e.\ $x \in \Omega \ \Rightarrow \ K(v,x) \geq \Lambda$ for every $x \in \Omega$,}
\end{equation}
where the function $K$ is defined by
\begin{equation}\label{slod2}
K(v,x) := \limsup_{\veps \to 0} 2 \veps^{-2} \sup_{|y| = 1} \left\{ v(x + \veps y) - v(x) - \veps \langle Dv(x), y \rangle \right\}
\end{equation}
at points $x$ where $v$ is differentiable and $K(v,x) := + \infty$ otherwise. For $v$ twice differentiable in $x$ one has
$$
    K(v,x) = \lambda_N(D^2 v(x)),
$$
the largest eigenvalue of the (symmetrized) Hessian of $v$ in $x$. We note that Slodkowski's theorem plays the role of {\em Jensen's lemma} in classical approaches to comparison theorems for viscosity solutions (see Lemma 3.10 of Jensen \cite{Je88} and Lemma A.3 of \cite{CrIsLiPL92}). In fact, these two results are in some sense equivalent as described in Harvey-Lawson \cite{HL13}.

\begin{proof}[Proof of Lemma \ref{SA_ae}]

For $\Lambda \geq 0$ and $x_0 \in \Omega$, set $w_{\Lambda;x_0} := w + \Lambda Q_{x_0}$ with $Q_{x_0}$ defined by \eqref{Qx0}. At almost every $x \in \Omega$, $w$ is twice differentiable and one has
$$
    D^2 w(x) \in \widetilde{\mathcal{P}} \ \Leftrightarrow \  D^2 w_{\Lambda;x_0}(x) \in \widetilde{\mathcal{P}} + \Lambda I \ \Leftrightarrow \ K(w_{\Lambda;x_0}, x) \geq \Lambda.
$$
Using Slodkowski's theorem \eqref{slod1} one concludes that for each $\Lambda \geq 2 \lambda$ and each $x_0 \in \Omega$
\begin{equation}\label{sat5}
     \mbox{$K(w_{\Lambda;x_0}, x) \geq \Lambda$ for every $x \in \Omega$.}
\end{equation}

Assume, by contradiction that $w \not\in \PSHD(\Omega)$. Hence there exists $x_0 \in \Omega$ and a triple $(\veps, r, a)$ such that
\begin{equation}\label{sat6}
     \mbox{$(w - a)(x_0) = 0$ \quad and \quad $(w - a)(x) \leq - \frac{\veps}{2} |x - x_0|^2, \ \ \forall \ x \in B_r(x_0)$.}
\end{equation}
Since $w - a$ is $2\lambda$-semi-convex and has a local maximum in $x_0$, one knows that $w-a$ is differentiable in $x_0$ and hence $w$ is differentiable in $x_0$ and $Dw(x_0) = Da(x_0)$ and hence from \eqref{sat6} one has also
\begin{equation}\label{sat7}
     a(x) = w(x_0) + \langle Dw(x_0), x - x_0 \rangle.
\end{equation}
Hence $w_{\Lambda;x_0}$ is differentiable in $x_0$ for each $\Lambda \geq 2 \lambda$ and therefore $K(w_{\Lambda;x_0}, x_0)$ is defined by \eqref{slod2} with
$$
    K(w_{\Lambda;x_0}, x_0) = \limsup_{\veps \to 0} 2 \veps^{-2} \sup_{|y| = 1} \left\{ w(x_0 + \veps y) - a(x_0 + \veps y) + \frac{\Lambda}{2} |\veps y|^2 \right\}.
$$
by using \eqref{sat7}. However, by \eqref{sat6} one then has
$$
 K(w_{\Lambda;x_0}, x_0) \leq \limsup_{\veps \to 0} 2 \veps^{-2} \sup_{|y| = 1} \left\{ \frac{\Lambda - \veps}{2} |\veps y|^2 \right\} = \Lambda - \veps,
$$
which contradicts \eqref{sat5}.
\end{proof}
This completes Step 1.

\vspace{2ex}

\noindent {\bf Step 2:} {\em Approximation by the sup-convolution and uniform upper semicontinuity of the elliptic map.}

In order to pass from the special case of $u$ and $v$ semi-convex to the general case of $u$ and $v$ upper semicontinuous, we will make use of the standard device of the sup-convolution and exploit the uniform upper semicontinuity of the elliptic maps to pass to the limit. We first recall the definition of the approximation. Given $u$ an upper semi-continuous function on a bounded domain $\Omega$ satisfying $|u| \le M$ for some $M \ge 0$ and given $\veps > 0$, the {\em sup-convolution} $u^{\veps}$ is defined by
\begin{equation}\label{defn:sup_conv}
u^\veps(x) = \sup_{z \in \R^N} \left\{ u(x-z) - \frac{1}{\veps}|z|^2 \right\} \quad \forall x \in \Omega,
\end{equation}
where one extends $u$ to be $-\infty$ outside of $\Omega$. The function defined in \eqref{defn:sup_conv} satisfies the following well-known properties (cf.\ Theorem 8.2 of \cite{HL09}, for example):
$$
    \mbox{$u^\veps$ decreases to $u$ as $\veps \rightarrow 0$}
$$
and
$$
    \mbox{$u^\veps$ is $\frac{1}{\veps}$-semi-convex.}
$$
The following lemma says that while the sup-convolution approximation $u^\veps$ of a $\Theta$-subharmonic function $u$ may not be $\Theta$-subharmonic, if $\Theta$ is uniformly upper semicontinuous, then an arbitrarily small quadratic perturbation of $u^\veps$ is $\Theta$-subharmonic.

\begin{lem}\label{lem:sup_conv} If $\Theta$ is a uniformly upper semicontinuous elliptic map on $\overline{\Omega}$ and if $u \in \TSH(\Omega)$ with $|u| \le M$ on $\Omega$ then for every $\eta > 0$ there exists $\overline{\veps} = \overline{\veps}(\eta) > 0$ such that
$$
    u^\veps( \cdot) + \eta | \cdot |^2 \in \TSH(\Omega_{\delta}), \ \ \forall \ \veps \in (0, \overline{\veps}(\eta)],
$$
where $\Omega_{\delta}$ is defined by \eqref{Omega_delta} and $\delta = \sqrt{2 \veps M}$.
\end{lem}

\begin{proof} For each $\eta > 0$ there exists $\delta = \delta(\eta, u) > 0$ such that
\begin{equation}\label{sc1}
 u(\cdot - z) + \eta | \cdot |^2 \in \TSH(\Omega_{\delta}), \ \ \forall \ z \in B_{\delta}(0),
\end{equation}
since $\Theta$ is uniformly upper semicontinuous and hence one can apply Proposition \ref{WTP}. Consider the following family of functions on $\Omega_\delta$
$$
 \mathcal{F}_z = \left\{ u(\cdot-z) - \frac{1}{\veps}|z|^2 + \eta |\cdot|^2 : |z| < \delta \right\}.
$$
One has $\mathcal{F}_z \subset \Theta \mathrm{SH}(\Omega_\delta)$ by \eqref{sc1} and the affine property (A) of Lemma \ref{max_affine}. By the sup property (S) of Lemma \ref{limit_sup},
$$
    u^\veps + \eta | \cdot |^2 := \sup_{w \in \mathcal{F}_z} w
$$
admits an upper semicontinuous regularization which is $\Theta$-subharmonic on $\Omega_\delta$ if $\veps \le \frac{\delta^2}{2M}$.
In addition, $u^\veps + \eta | \cdot|^2$ is also continuous since it is quasi-convex and hence agrees with its regularization.
\end{proof}

\vspace{2ex}

\noindent {\bf Step 3:} {\em Proof of the subaffine theorem in the general case.}

Since the result is local, we can assume that $u \in \TSH(\Omega)$ and $v \in \TSHD(\Omega)$ are bounded from above by using their upper semicontinuity.  We may also assume that they are bounded from below, by replacing $u$ and $v$ with $u_m = \max \{ u, -m \}$ and $v_m = \max \{ v, -m \}$, if need be, and then taking the limit of $u_m + v_m$ as $m \rightarrow +\infty$. In fact, if $\{ w_m \}_{m \in \N} \subset \USC(\Omega) \cap \SA(\Omega)$ has a decreasing limit $w_m \searrow w$ in $\Omega$, then $w \in \USC(\Omega) \cap \SA(\Omega)$ where the $\USC(\Omega)$ claim is standard and the $\SA(\Omega)$ claim comes from the decreasing limits property (L) of Lemma \ref{limit_sup} applied to the constant elliptic map $\Theta = \widetilde{\mathcal{P}}$ and the fact that $w \in \SA(\Omega)$ if and only if $w \in \widetilde{\mathcal{P}} \mathrm{SH} (\Omega)$ as noted in \eqref{sat3}.

We consider  an arbitrary decreasing sequence $\{\eta_j\}$ of positive real numbers such that $\eta_j \rightarrow 0$.  Let $u_j = u^{\veps_j} + \eta_j | \cdot| ^ 2$ be given as in Lemma \ref{lem:sup_conv} above and take $\veps_j = \min\{\eta_j, \bar{\veps}(\eta_j)\}$.
We have that $u_j \in \Theta(\Omega_{\delta_j})$ where $\delta_j = \sqrt{2 \veps_j M}$ and that $u_j$ is semi-convex. Moreover, the sequence $u_j$ decreases to $u$ as $j \rightarrow +\infty$. We may construct a sequence $v_j$ in the same way. By the subaffine theorem for semi-convex functions of Step 1, we have $u_j + v_j \in \SA(\Omega_{\delta_j})$.
Since $u_j + v_j \searrow u + v$ and $\Omega_{\delta_j}$ exhausts $\Omega$ as $\delta_j \searrow 0$ decreases, we obtain $u + v \in \SA(\Omega)$ by property the decreasing limit property (L) of Lemma \ref{limit_sup}.

This completes Step 3 and hence the proof of Theorem \ref{thm:SAT}. \hfill $\square$

\section{Elliptic cones and boundary convexity}\label{bdy_convexity}

In this section, our main aim is to clarify what will be an admissible domain for the Dirichlet problem for weakly $\Theta$-harmonic functions. Again following the path laid down by Harvey-Lawson \cite{HL09} in the case of elliptic sets (constant elliptic maps), what is needed is a suitable notion of strict convexity of $\partial \Omega$ with respect to $\Theta$. In all that follows, $\Omega$ be a bounded domain in $\R^N$ with $C^2$ boundary; that is, for each $x \in \partial \Omega$ there exists $\rho \in C^2(B_r(x))$ for some $r > 0$ such that
\begin{equation}\label{LDF}
\mbox{$\Omega \cap B_r(x) = \{ y \in B_r(x): \rho(y) < 0 \}$ and $D \rho \neq 0$ on $B_r(x)$.}
\end{equation}
The convexity notion will be formulated in terms of the following structure.
\begin{defn}\label{defn:ell_cone} An {\em elliptic cone} \footnote{Elliptic cones are called {\em Dirichlet ray sets} in \cite{HL09}.}  $\EC{\Phi}$ is an elliptic subset which is also a pointed cone; that is $\EC{\Phi} \subset \Ss(N)$ is closed, non empty, proper and satisfies $ \EC{\Phi} + \mathcal{P} \subset \EC{\Phi}$ as well as
\begin{equation}\label{EC1}
 A \in \EC{\Phi} \Rightarrow tA \in \EC{\Phi} \quad \text{for all $t \ge 0$}.
 \end{equation}
\end{defn}
Since the cone $\EC{\Phi}$ is also an elliptic subset, the property \eqref{EC1} implies
\begin{equation}\label{EC2}
 A \in \EC{\Phi}^{\circ} \Rightarrow tA \in \EC{\Phi}^{\circ} \quad \text{for all $t > 0$}.
\end{equation}
\begin{defn}\label{bdy_convexity}
Let $\EC{\Phi}$ be an elliptic cone. The boundary $\partial \Omega$ will be said \emph{strictly $\EC{\Phi}$-convex in $x \in \partial \Omega$} if there exists a local defining function $\rho$ for $\partial \Omega$ in $x$ such that
$$
    D^2 \rho(x) |_{T_x \partial \Omega} = B|_{T_x \partial \Omega} \quad \text{for some $B \in \EC{\Phi}^{\circ}$},
$$
where $T_x \partial \Omega$ is the tangent space to $\partial \Omega$ at $x$.
\end{defn}
As shown in Lemma 5.2 of \cite{HL09}, this convexity notion\footnote{In pace of a local defining function $\rho$, one can formulate the convexity notion in terms of the second fundamental form $\Pi$ of $\partial \Omega$ as shown in Corollary 5.4 of \cite{HL09}.} is independent of the choice of the local defining function $\rho$ various equivalent ways, including . Indeed, if $\psi$ is another local defining function, then $\psi = w \rho$ for some smooth function $w$ with $w(x) > 0$. Using \eqref{LDF}, one has $\rho(x) = 0$ and $D \rho(x)$ belongs to the normal space $N_x \partial \Omega$ at $x$ and hence
$$
D^2 \psi |_{T_x \partial \Omega} = w(x) D^2 \psi |_{T_x \partial \Omega} = w(x) B|_{T_x \partial \Omega}.
$$
Since $w(x) > 0$, the property \eqref{EC2} ensures that $w(x) B \in \EC{\Phi}^{\circ}$.

This boundary convexity property will be used to construct suitable barrier functions in the treatment of the Dirichlet problem for weakly $\Theta$-harmonic functions provided that $\Theta$ is associated to a suitable elliptic cone $\EC{\Theta}$. For a constant elliptic map $\Theta$ (namely for an elliptic set $\Phi$), the associated cone is defined in the following way. 

\begin{defn}\label{associated_cone} Let $\Phi \in \mathcal{E}$ be an elliptic set and $B \in \Ss(N)$. The elliptic cone $\EC{\Phi}$ associated to $\Phi$ is defined by
\begin{equation}\label{defn:ass_cone}
\mbox{ $\EC{\Phi}:= \overline{ \{ A \in \Ss(N): \ \exists t_0 = t_0(A) \in \R \ \text{such that} \ B + tA \in \Phi, \forall \ t \geq t_0 \} }$.}
\end{equation}
\end{defn}

One knows from Section 5 of \cite{HL09} that $\EC{\Phi}$ is an elliptic cone and that this construction does not depend on the vertex $B \in \Ss(N)$ since one has passed to the closure in $\Ss(N)$ in the definition \eqref{defn:ass_cone}. Using \eqref{defn:ass_cone} with $B = 0$ it is easy to see that
\begin{equation}\label{cone_exs}
    \mbox{ $\EC{\mathcal{P}} = \mathcal{P}$ \quad and \quad $\EC{\mathcal{\widetilde{P}}} = \widetilde{\mathcal{P}},$}
\end{equation}
which yield useful elementary examples.

Elliptic maps $\Theta$ yield an {\em elliptic cone map} by using \eqref{defn:ass_cone} pointwise; that is,
$$
    \EC{\Theta}(x) := \EC{\Theta(x)}.
$$
If the elliptic map $\Theta$ is uniformly upper semicontinuous, this cone map is constant. By Proposition \ref{uuscd}, the same is true for the dual elliptic cone map $\EC{\widetilde{\Theta}}$ associated to $\widetilde{\Theta}$. 

\begin{prop}\label{uusc_cone} If $\Theta$ is a uniformly upper semicontinuous elliptic map on $\Omega$ then there exists an elliptic cone $\EC{\Theta}$ such that
$$
    \mbox{ $\EC{\Theta}(x) = \EC{\Theta}$ for each $x \in \overline{\Omega}$.}
$$
Moreover,
$$
    \mbox{ $\EC{\Theta} = \Theta(x)$ for each $x \in \overline{\Omega}$ such that $\Theta(x)$ is an elliptic cone;}
$$
that is, for each $x \in \overline{\Omega}$ such that $A \in \Theta(x) \Leftrightarrow tA \in \Theta(x)$ for each $t \geq 0$.
\end{prop}

\begin{proof} In order to show that the cone map associated to $\Theta$ is constant, one makes use of the following characterization of the interior of the associated cone $\EC{\Theta}(x)$:
\begin{equation}\label{CM3}
   \mbox{ $A \in \EC{\Theta}(x)^{\circ} \Leftrightarrow \exists \ \veps > 0$ and $R > 0$ such that $t(A - \veps I) \in \Theta(x), \ \forall \ t \geq R$.}
\end{equation}
This claim in proven in Corollary 5.10 of \cite{HL09}. Moreover, if $A_0 \in \Theta(x)^{\circ}$ then there exist $\veps_0 > 0$ and $R_0 > 0$ such that \eqref{CM3} holds for each $A$ in a neighborhood of $A_0$ and each pair $R \geq R_0, \veps \leq \veps_0$. Now, with $x \in \overline{\Omega}$ fixed and $A \in \EC{\Theta}(x)^{\circ}$ the characterization \eqref{CM3} yields
\begin{equation}\label{CM4}
   \mbox{ $\exists \ \bar{t} = \bar{t}(A)$ such that $tA \in \Theta(x)$ for every $t \geq \bar{t}$.}
\end{equation}
In fact, one picks $\bar{t}(A) = R$ and then $tA = tA - \veps tI + \veps tI \in \Theta(x)$ for each $t \geq R$. By the uniform upper semicontinuity of the elliptic map $\Theta$, for fixed $\veps > 0$
there exists $\delta = \delta(\veps) > 0$ such that
$$
    \mbox{$tA + \veps I \in \Theta(y)$ if $y \in B_\delta(x) \cap \overline{\Omega}$ and $t \geq \bar{t}$}.
$$
Hence one has $ A \in \EC{\Theta}(y)$ by the definition \eqref{defn:ass_cone}. Since $A \in \EC{\Theta}(x)^{\circ}$ is arbitrary, by passing to the closure, one obtains
$$
    \mbox{ $\EC{\Theta}(x) \subset \EC{\Theta}(y)$ for all $x,y \in \overline{\Omega}$ such that $|x-y| < \delta$}.
$$
By inverting the roles of $x$ and $y$, one obtains that $\EC{\Theta}$ is locally constant on $\Omega$, which is connected.

For the second assertion, suppose that $x \in \overline{\Omega}$ is such that $\Theta(x) \in \mathcal{E}$ is itself a cone; that is, if condition \eqref{EC1} holds, then
\begin{equation}\label{CM2}
    \EC{\Theta}(x) = \Theta(x).
\end{equation}
Indeed, one sees that $\Theta(x) \subset \EC{\Theta}(x)$ by taking $B = 0$ and $t_0=0$ in \eqref{defn:ass_cone}. Conversely, if $A_0 \in \EC{\Theta}(x)$ then it is a limit in $\Ss(N)$ of $A_k \in \Ss(N)$ such that $tA_k \in \Theta(x)$ for each $t \geq t_k$ with $t_k = t_0(A_k)$. If $t_k \leq 1$ then $A_k \in \Theta(x)$. If, on the other hand $t_k > 1$, then $t_k A_k \in \Theta(x)$ and using the cone property for $\Theta(x)$ one again has $A_k \in \Theta(x)$ by selecting $t = 1/t_k$. Hence the claim \eqref{CM2} follows since $\Theta(x)$ is closed.

\end{proof}

We will now state the main result of this section, which ties $\EC{\Theta}$-convexity to the admissibility of $\partial \Omega$ for the Dirichlet problem through the existence of $\EC{\Theta}$-subharmonic {\em global defining functions} $\rho$ for the boundary; that is, $\rho \in C^2(\overline{\Omega})$ such that $\Omega = \{ \rho < 0\}$, $D \rho \neq 0$ on $\partial \Omega$ and $\rho$ is $\EC{\Theta}$-subharmonic on $\overline{\Omega}$.

\begin{thm}\label{thm:convexity} Let $\EC{\Theta}$ be an elliptic cone. If $\partial \Omega$ is strictly $\EC{\Theta}$-convex at each $x \in \partial \Omega$, there exists a global defining function $\rho$ for $ \partial \Omega$ that is strictly $\EC{\Theta}$-subharmonic on $\overline{\Omega}$.
In particular, if $\partial \Omega$ is strictly $\EC{\Theta}$-convex and $\EC{\Theta}$ is the elliptic cone associated to a uniformly upper semicontinuous elliptic map $\Theta$ on $\overline{\Omega}$, then
\begin{equation}\label{eq:crho}
\exists \veps > 0, R>0 \text{ such that } C(\rho - \veps|x|^2/2) \text{ is $\Theta$-subharmonic on $\Omega$ for all $C \ge R$}.
\end{equation}
\end{thm}
\begin{proof} When $\Theta$ is an elliptic set (a constant elliptic map), this is Theorem 5.12 of \cite{HL09} and so we have the existence of $\rho$ since $\EC{\Theta}$ is constant by Proposition \ref{uusc_cone}.  When $\Theta$ is a non constant elliptic map, the proof of \eqref{eq:crho} requires some additional reasoning. By the characterization \eqref{CM4}, for every $x \in \overline{\Omega}$ there exists $\veps_x > 0$ e $C_x > 0$ such that
$$
    C(A-3 \veps I) \in \Theta(x)^{\circ}
$$
for all $C \ge C_x$, $\veps \le \veps_x$ and $A$ in a neighborhood of $D^2 \rho (x)$. Since $\rho \in \mathcal{C}^2(\overline{\Omega})$ one has that $C(D^2 \rho(y) -2 \veps I) \in \Theta(x)^{\circ}$ for every $y$ in a neighborhood $U_x$ of $x$. By the uniform upper semicontinuity of $\Theta$ (and shrinking $U_x$, if necessary) one obtains $C(D^2 \rho(y) -2 \veps I) + C_x \veps_x I \in \Theta(y)$ on $U_x$. Hence $C(D^2 \rho(y) - \veps I) \in \Theta(y) + C\veps I - C_x \veps_x I$. By extracting a finite subcovering $\{ U_{x_i} \}_{i =1}^N$ of $\overline{\Omega}$, and by taking the minimum $\bar{\veps}$ of $\{ \veps_{x_i} \}_{i=1}^N$ and the maximum $C$ of $\{C_{x_i} \}_{i=1}^N$, one obtains that $C(D^2 \rho(y) - \bar{\veps} I) \in \Theta(y) + C\bar{\veps} I - C_x \veps_x I \subset \Theta(y)$ on $\overline{\Omega}$ with $C >> 0$.
\end{proof}

\section{Perron's method and the Dirichlet problem for elliptic maps}\label{DirProb}

In this section, we give the main result of this paper on the existence and uniqueness of solutions to the Dirichlet problem for weakly $\Theta$-harmonic functions where $\Theta$ is a uniformly upper semicontinuous elliptic map, which generalizes Theorem 6.2 of \cite{HL09} for constant elliptic maps. The proof is an implementation of the Perron method where most of the work has already been performed in the previous sections. When $\Theta$ corresponds to an elliptic branch of a fully nonlinear PDE \eqref{FNE}, one has an existence and uniqueness result for viscosity solutions with restrictions as described in Section \ref{viscosity_solns}. Examples and applications will be discussed in the following section.

The main result is the following theorem, where we recall that $\Theta$ uniformly upper semicontinuous on $\Omega$ extends to $\overline{\Omega}$ by Proposition \ref{extension}.

\begin{thm}\label{thm:EU}
Let $\Omega \subset \R^N$ be a bounded domain with $\partial \Omega$ of class $C^2$ and let $\Theta$ be a uniformly upper semicontinuous elliptic map on $\overline{\Omega}$. Suppose that $\partial \Omega$ is both strictly $\EC{\Theta}$-convex and strictly $\EC{\widetilde{\Theta}}$-convex. Then for each $\varphi \in C(\partial \Omega)$, there exists a unique $u \in C(\overline{\Omega})$ which is a $\Theta$-harmonic in the sense of Definition \ref{def:Weak_SH} and such that  $u = \varphi$ on $\partial \Omega$.
\end{thm}

\begin{proof} \textit{(Uniqueness)}. This follows immediately from the comparison principle of Theorem \ref{thm:CP}. In fact, if $u,v$ are two solutions such that $u = v = \varphi$ on $\partial \Omega$ one has both $u \leq v$ and $v \leq u$ on $\partial \Omega$, where $u,v \in \TSH$ and $-u,-v \in \TSHD$.

\textit{(Existence)}. Consider the Perron family for the given boundary data $\varphi$
\begin{equation}\label{PF1}
    \F(\varphi) := \{ v \in \USC(\overline{\Omega}) : v|_{\Omega} \in \TSH \text{ and } v|_{\partial \Omega} \le \varphi\}.
\end{equation}
We will show that the upper envelope
\begin{equation}\label{PF2}
    u(x) := \sup_{v \in \F(\varphi)} v(x) \qquad \forall x \in \overline{\Omega},
\end{equation}
is continuous, $\Theta$-harmonic and satisfies the boundary condition.

\vspace{2ex}

\noindent {\bf Step 1:} {\em The Perron family \eqref{PF1} is uniformly bounded from above on $\overline{\Omega}$ and hence the Perron function \eqref{PF2} is well defined and finite on $\overline{\Omega}$.}

The proof of this fact uses the maximum principle for subaffine functions (Lemma \ref{MPSA}) and the following consequence of having uniform upper semicontinuity
of an elliptic map on a compact set.

\begin{lem}\label{lem:SHQF} Let $\Theta$ be a uniformly upper semicontinuous elliptic map on $\overline{\Omega}$. Then there exist $B, \widetilde{B} \in \Ss(N)$ such that
$$
    B \in \bigcap_{x \in \overline{\Omega}} \Theta(x) \quad \text{and} \quad \widetilde{B} \in \bigcap_{x \in \overline{\Omega}} \widetilde{\Theta}(x).
$$
In particular, $B$ and $\widetilde{B}$ can be taken to be a positive multiple of the identity matrix $I$.
\end{lem}

\begin{proof} Since uniform upper semicontinuity and ellipticity are preserved under taking duals, it is enough to consider the existence of $B$. Using the uniform upper semicontinuity with $\veps = 1$, the formulation \eqref{uusc2} yields
\begin{equation}\label{SHQF2}
 \mbox{   $\Theta(y) + I \subset \Theta(x)$ for each $x,y \in \overline{\Omega}$ with $|x-y| < \delta$.}
\end{equation}
Since $\Theta(y)$ is an elliptic set for each $y \in \overline{\Omega}$, there exists $t_y > 0$ such that
\begin{equation}\label{SHQF3}
 \mbox{   $tI \in \Theta(y)$ for each $t \geq t_y$.}
\end{equation}
For a proof of this elementary fact, see Section 3 of \cite{HL09}. Since $\overline{\Omega}$ is compact there exists a finite open covering $\{ B_{\delta}(y_k) \}_{k=1}^n$ with $y_k \in \overline{\Omega}$. Combining \eqref{SHQF2} and \eqref{SHQF3}, for each $x \in \overline{\Omega}$
$$
    \mbox{   $(1 + t)I \in \Theta(x)$ provided that  $\displaystyle{t \geq T:= \max_{1 \leq k \leq n}t_k}$.}
$$
Hence $B:=(1+T)I$ gives the required matrix.
\end{proof}

Using the lemma and the coherence property one has $Q_{\widetilde{B}} \in C^2(\overline{\Omega}) \cap \TSHD(\Omega)$ where $Q_{\widetilde{B}}(x):= \frac{1}{2} \langle \widetilde{B}x, x \rangle $. Hence for each $v \in \mathcal{F}(\varphi) \subset \USC(\overline{\Omega}) \cap \TSH(\Omega)$ one has
$$
    v + Q_{\widetilde{B}} \in \USC(\overline{\Omega}) \cap \SA(\Omega).
$$
Applying the maximum principle of Lemma \ref{MPSA} and using the nonnegativity of $Q_{\widetilde{B}}$ plus the boundary condition on $v$ one has
$$
    v \leq v + Q_{\widetilde{B}} \leq \max_{\partial \Omega}(v + Q_{\widetilde{B}}) \leq \max_{\partial \Omega}(\varphi + Q_{\widetilde{B}}) \leq M \quad \text{on} \ \overline{\Omega},
$$
for some $M \in \R$ as $\varphi$ are $Q_{\widetilde{B}}$ bounded. This completes Step 1.

\vspace{2ex}

\noindent {\bf Step 2:} {\em The Perron function $u$ belongs to the Perron family $\mathcal{F}(\varphi)$; that is}
$$
    u \in \USC(\overline{\Omega}), \quad u_{|\Omega} \in \TSH(\Omega), \quad \emph{and} \quad u_{|\partial \Omega} \leq \varphi.
$$
When $\Theta$ is a constant elliptic map, this is Proposition 6.7 of \cite{HL09} and the proof generalizes easily to the case of $\Theta$ uniformly upper semicontinuous on $\overline{\Omega}$. For completeness, we will give the proof. The main ingredients are the existence of suitable barrier functions (following from the $\EC{\Theta}$-convexity and the $\EC{\widetilde{\Theta}}$-convexity of the boundary) and the maximum principle for subaffine functions. The properties (M), (A), and (S) of the space $\TSH$ are needed.

By the property (S) of Lemma \ref{limit_sup}, the upper semicontinuous regularization of $u$ satisfies
\begin{equation}\label{P1}
u^* \in \TSH(\Omega) \cap \USC(\overline{\Omega}).
\end{equation}
We claim that $\Omega$ strictly $\EC{\widetilde{\Theta}}$-convex implies that
\begin{equation}\label{P2}
u^*_{|\partial \Omega} \leq \varphi.
\end{equation}
When $\Theta$ is constant, this claim is Lemma 6.8 of \cite{HL09} and we recall that $\widetilde{\Theta}$ uniformly upper semicontinuous implies that the cone map $\EC{\widetilde{\Theta}}$ is constant. Hence the proof carries over without change,  but in order to illustrate the role of the convexity assumption we will reproduce it here. Theorem \ref{thm:convexity} applied to $\widetilde{\Theta}$ yields a smooth global defining function $\rho$ which is strictly $\EC{\widetilde{\Theta}}$-subharmonic on $\overline{\Omega}$. In particular, for $x_0 \in \partial \Omega$ fixed but arbitrary, combining the affine property (A) of Lemma \ref{max_affine} with \eqref{eq:crho} yields the existence of $\veps > 0$ and $R> 0$ so that
\begin{equation}\label{P3}
    C(\rho - \veps|x-x_0|^2) \text{ is $\widetilde{\Theta}$-subharmonic on $\Omega$ for all $C \ge R$}.
\end{equation}
Since $\rho$ is continuous and vanishes on $\partial \Omega$, for each $\delta > 0$ one can pick $C$ in \eqref{P3} large enough to ensure
\begin{equation}\label{P4}
\mbox{ $\varphi(x) + C(\rho - \veps|x-x_0|^2) = \varphi(x) -C \veps|x-x_0|^2 \leq \varphi(x_0) + \delta$ for all $x \in \partial \Omega$.}
\end{equation}
By \eqref{P3} and \eqref{P4}, for each $v \in \mathcal{F}(\varphi) \subset \USC(\overline{\Omega}) \cap \TSH(\Omega)$ one has
\begin{equation}\label{P5}
    w:= v + C(\rho - \veps|x-x_0|^2) \in \USC(\overline{\Omega}) \cap \SA(\Omega) \quad with \quad w_{|\partial \Omega} \leq \varphi(x_0) + \delta.
\end{equation}
Applying again the maximum principle of Lemma \ref{MPSA} to $w$, \eqref{P5} yields
\begin{equation}\label{P6}
\mbox{ $w(x) = v + C(\rho - \veps|x-x_0|^2) \leq \varphi(x_0) + \delta$ for each $x \in \overline{\Omega}$.}
\end{equation}
Taking the sup over $v \in v \in \mathcal{F}(\varphi)$ and then performing the upper semicontinuous regularization in \eqref{P6} yields
$$
    \mbox{ $u^*(x) + C(\rho - \veps|x-x_0|^2) \leq \varphi(x_0) + \delta$ for each $x \in \overline{\Omega}$,}
$$
which when evaluated in $x = x_0$ yields the claim \eqref{P2}, as $x_0$ was arbitrary.

Combining \eqref{P1} and \eqref{P2} shows that $u^* \in \mathcal{F}(\varphi)$ and hence $u^* \leq u$ on $\overline{\Omega}$ by the definition of $u$, but one always has $u \leq u^*$ by the definition of $u^*$ and hence
$$
    u= u^* \in \mathcal{F}(\varphi),
$$
as desired. This completes Step 2.

\vspace{2ex}

\noindent {\bf Step 3:} $-u_{|\Omega} \in \TSHD(\Omega)$

For $\Theta$ a constant elliptic map, this claim is Lemma 6.12 of \cite{HL09} and its proof is by contradiction by making use of a ``bump construction'' which exploits the property (M) of Lemma \ref{max_affine}. Indeed, suppose that $-u|_\Omega \notin \TSHD(\Omega)$ and hence there exist $x_0 \in \Omega$ and $A \in \widetilde{\widetilde{\Theta}}(x_0) = \Theta(x_0)$ such that $-u + Q_{A,x_0}$ is not subaffine in $x_0$. With the aid of the quadratic form $Q_{A,x_0}$, one can construct the function
$$
w = \left\{
\begin{array}{ll}
u & \text{ on } \overline{\Omega} \setminus B_r(x_0) \\
\max\{u, Q_{A,x_0}+a+\veps|x-x_0|^2\} & \text{ on }  B_r(x_0)
\end{array}
\right.
$$
for a suitable affine function $a$ and suitable $\veps, r > 0$ so that $w$ satisfies $w(x_0) > u(x_0)$. However, $w$ belongs to the family $\F(\varphi)$ and hence satisfies $w(x_0) \le u(x_0)$, a contradiction. This completes Step 3.

\vspace{2ex}

\noindent {\bf Step 4:} {\em $u$ is continuous at each point $x_0 \in \partial \Omega$ and $u|_{\partial \Omega} = \varphi$}

Since $u \in \USC(\overline{\Omega})$ and $u|_{\partial \Omega} \leq \varphi$ by Step 1, it suffices to show that
\begin{equation}\label{P9}
\mbox{ $\displaystyle{\liminf_{x \to x_0} u(x) \geq \varphi(x_0)}$ \quad for each $x_0 \in \partial \Omega$,}
\end{equation}
which obviously involves another barrier argument exploiting the assumption that $\partial \Omega$ is also strictly $\EC{\Theta}$-convex (as done in Lemma 6.9 of \cite{HL09}) for constant elliptic maps. As in the proof of claim \eqref{P2} from Step 2, for $x_0 \in \partial \Omega$ fixed but arbitrary, combining the affine property (A) of Lemma \ref{max_affine} with \eqref{eq:crho} yields the existence of $\veps > 0$ and $R> 0$ so that
\begin{equation}\label{P10}
    C(\rho - \veps|x-x_0|^2) \text{ is $\Theta$-subharmonic on $\Omega$ for all $C \ge R$}.
\end{equation}
For each $\delta > 0$ one can pick $C$ in \eqref{P10} large enough to ensure
\begin{equation}\label{P11}
\mbox{ $\varphi(x_0) + C(\rho - \veps|x-x_0|^2) \leq \varphi(x) + \delta$ for all $x \in \partial \Omega$.}
\end{equation}
The function $v:= \varphi(x_0) + C(\rho - \veps|x-x_0|^2) - \delta$ is smooth and belongs to the Perron family $\mathcal{F}(\varphi)$ by \eqref{P10} and \eqref{P11}. Therefore $v \leq u$ on $\overline{\Omega}$ by the definition of $u$ and hence
$$
    \liminf_{x \to x_0} u(x) \geq \liminf_{x \to x_0} v(x) = \varphi(x_0) - \delta.
$$
Since $\delta$ is arbitrary, the claim \eqref{P9} follows. This completes Step 4 and one last step remains.

\vspace{2ex}

\noindent {\bf Step 5:} {\em $u$ is continuous at each point $x \in \Omega$}

This claim is Proposition 6.11 of \cite{HL09} in the case of constant elliptic maps. The original argument of Walsh (\cite{W68}) used in \cite{HL09} makes effective use the following {\em strong translation property}
$$
    \mbox{ $u \in \TSH(\Omega) \Rightarrow u_y:= u(\cdot + y) \in \TSH(\Omega)$ for each small $y \in \R^N$,}
$$
which does not hold for a general uniformly upper semicontinuous elliptic map. However, the argument can be adapted to the situation in which only the {\em weak translation property} of Proposition \ref{WTP} is available. Indeed, fix $\veps > 0$ and set $\bar{\veps} = \veps/d^2$ with $d := \sup_{x \in\Omega} |x|$. By Proposition \ref{WTP}, there exists $\delta > 0$ such that for every $y$ with $|y| < \delta$ we have
$$ u_y + \bar{\veps} |\cdot|^2 \in \Theta \mathrm{SH}(\Omega_{\delta}) \quad \mbox{and} \quad \Omega_{\delta}:= \{x \in \Omega : d(x, \partial \Omega) > \delta\},$$
where we extend $u$ to be $-\infty$ on $\R^n \setminus \overline{\Omega}$. Let $C_\delta = \{x \in \overline{\Omega} : d(x, \partial \Omega) < \delta\}$. Since $u \in C(\partial \Omega)$, by rescaling  $\delta$ if necessary, one has
$$
u_y \le u + \veps \quad \text{on $C_{2\delta}$ if $|y| < \delta$}.
$$
Then define the function
$$
 g_y = \max\{u_y + \bar{\veps} |x|^2 - 2\veps, u \}.
$$
which satisfies $g_y \in \Theta \mathrm{SH}(\Omega_\delta)$ by applying the properties (A) and (M) of Lemma \ref{max_affine}. One has $g_y = u$ on  $C_{2\delta}$ since $u_y + \bar{\veps} |x|^2 - 2\veps \le u$ on  $C_{2\delta}$ . Hence $g_y \in \USC(\overline{\Omega})$ is $\Theta$-subharmonic on all of $\Omega$ and satisfies $g_y \leq \varphi$ on $\partial \Omega$ and consequently $g_y \in \mathcal{F}(\varphi)$. Thus $g_y \le u$ on $\overline{\Omega}$ and
$$
 u_y + \bar{\veps} |x|^2 - 2\veps \le g_y \le u,
$$
and hence for any $y$ with $|y| < \delta$
$$
 u_y \le u - \bar{\veps} |x|^2 + 2\veps \le u + 2\veps.
$$
Making the change of variables $z = x+y$, one has
$$
 \text{if $z, x \in \overline{\Omega}$ and $|z-x| < \delta$, then $u(z) \le u(x) + 2\veps$}.
$$
Interchanging the roles of $x$ and $y$, by exploiting the symmetry in the definition of continuity for $\Theta$, one obtains $|u(x) - u(z)| \le 2 \veps$.
\end{proof}

\section{Applications to fully nonlinear PDEs}\label{examples}

Having established the well-posedness of the Dirichlet problem for uniformly upper semicontinuous elliptic maps $\Theta$ on domains $\Omega$ which are suitably convex with respect to $\Theta$, we address the problem of obtaining well-posedness of the Dirichlet problem for elliptic branches of the nonlinear PDE \eqref{FNE}. Using Theorem \ref{thm:EU} as an abstract tool, we return to the class of fully nonlinear PDEs introduced in Proposition \ref{pick_branch} for which we are able to associate an elliptic map $\Theta$. We will give a fairly general structural condition on $F(x,A)$ which ensures the uniform upper semicontinuity of its associated elliptic map $\Theta$ and then give a description of the interiors of the elliptic cones $\EC{\Theta}^{\circ}$ and $\EC{\widetilde{\Theta}}^{\circ} $ used to define the needed boundary convexity (see Propositions \ref{UCbranch} and \ref{natural_cones} below). Combining these propositions with Theorem \ref{thm:EU}, when $F$ and $\Theta$ also satisfy the branch condition \eqref{def_branch2}, yields well-posedness results for viscosity solutions of an elliptic branch of the PDE. This theory will be illustrated with a few natural examples, including equations involving the eigenvalues of the Hessian $\lambda_k(D^2u)$ and their perturbations. Throughout we will attempt to make a comparison with known results using the classical viscosity approach including possibly reformulating the the given PDE with the aid of well chosen Bellman operators.

\subsection{Structure conditions, the comparison principle and admissible domains}

We begin with the structural condition on $F$ which will ensure the uniform upper semicontinuity of the associated elliptic map and hence the validity of the comparison principle.

\begin{prop}\label{UCbranch} Let $\Phi$ be a uniformly upper semicontinuous elliptic map on $\Omega$ and let $F \in C(\Omega \times \Ss(N), \R)$ satisfy the conditions \eqref{deg_ell} and \eqref{F_strict} of Proposition \ref{pick_branch} with respect to $\Phi$. Also assume that for some $\veps^* > 0$ one has the following condition:  for each $\veps \in (0, \veps^*]$ there exists $\delta = \delta(\veps) > 0$ such that
\begin{equation}\label{UCF}
F(y,A + \veps I) \geq F(x,A) \quad \forall A \in \Phi(x), \forall x,y \in \Omega \ {\rm such \ that \ } |x - y| < \delta.
\end{equation}
Then the elliptic map $\Theta$ defined by \eqref{def_branch1}; that is,
$$
\Theta(x) := \{ A \in \Phi(x):  F(x,A) \geq 0 \},
$$
extends to a uniformly upper semicontinuous elliptic map on $\overline{\Omega}$.
\end{prop}

\begin{proof} $\Theta$ gives an elliptic map by Proposition \ref{pick_branch} and will have the desired extension by Proposition \ref{uusc}, provided that $\Theta$ is uniformly upper semicontinuous on $\Omega$. Given $\veps^* > 0$ so that \eqref{UCF} holds, the degenerate ellipticity \eqref{deg_ell} implies that \eqref{UCF} continues to hold for each $\veps > \veps^*$ by taking $\delta(\veps) = \delta(\veps^*)$. Hence for any $\veps > 0$, let $x,y \in \Omega$ with $|x - y|< \delta$ and take any $B \in \Theta(x) + \veps I$ so that $B - \veps I \in \Theta(x) \subset \Phi(x)$. Using the definition of $\Theta$ and \eqref{UCF} one finds
$$
0 \leq F(x, B - \veps I) \leq F(y, B - \veps I + \veps I) = F(y, B),
$$
so that $B \in \Theta(y)$ provided that $B \in \Phi(y)$. Using the uniform upper semicontinuity of $\Phi$, one has $B = (B - \veps I) + \veps I \in \Phi(y)$ for $x,y \in \Omega$ with $|x-y| < \delta_{\Phi}(\veps)$. Hence $\Theta(x) + \veps I \subset \Theta(y)$ for each $x,y \in \Omega$ with $|x - y|< \min\{ \delta(\veps), \delta_{\Phi}(\veps)\}$. Interchanging the roles of $x$ and $y$ gives the uniform upper semicontinuity of $\Theta$.
\end{proof}

Notice that if $F$ satisfies the hypotheses of Proposition \ref{UCbranch} then the associated elliptic map $\Theta$ and its dual $\widetilde{\Theta}$ will satisfy the comparison principle of Theorem \ref{thm:CP}. General and concrete examples of $F$ satisfying \eqref{UCF} will be given as well as a comparison with other known structural conditions which ensure the validity of the comparison principle (see Remark \ref{SC_NTD}, Example \ref{SC_CAB} and the discussion which follows).

We now give a description of the interiors of the elliptic cones $\EC{\Theta}^{\circ}$ and $\EC{\widetilde{\Theta}}^{\circ} $ used to specify the needed strict convexity of admissible domains $\Omega$ for the Dirichlet problem.

\begin{prop}\label{natural_cones} Let $\Phi$, $F$ be as in Proposition \ref{UCbranch} and $\Theta$ the associated uniformly upper semicontinuous elliptic map on $\overline{\Omega}$. Then, for any $\bar{x} \in \Omega$,
\begin{equation}\label{ENC1}
\EC{\Theta}^{\circ} = \{ A \in \EC{\Phi(\bar{x})}^\circ : \text{$\exists , \epsilon, R > 0$ s.t. $F(\bar{x}, C(A - \epsilon I)) \ge 0, \ \forall \ C \ge R$} \}
\end{equation}
and
\begin{equation}\label{ENC2}
\EC{\widetilde{\Theta}}^{\circ} = \EC{\widetilde{\Phi}}(\bar{x})^\circ \cup \{ A \in \Ss(N) : \text{$\exists \, \epsilon, R > 0$ s.t. $F(\bar{x}, C(-A + \epsilon I)) < 0, \ \forall \ C \ge R$} \}
\end{equation}
where $\EC{\Phi(x)}$ (resp. $\EC{\widetilde{\Phi}(x)}$) is the elliptic cone associated to $\Phi(x)$ (resp. $\widetilde{\Phi}(x)$).
\end{prop}

\begin{proof} The following equivalence, which holds for every $\Phi \in \mathcal{E}$, will be used throughout the proof:
\begin{itemize}
\item[\textit{i)}] $A \in \EC{\Phi}^{\circ}$.
\item[\textit{ii)}] There exist $\epsilon, R > 0$ such that $C(A-\epsilon I) \in \Phi$ for all $C \ge R$.
\item[\textit{iii)}] There exist $\hat{\epsilon}, R > 0$ such that $C(A-\hat{\epsilon}I) \in \Phi^\circ$ for all $C \ge R$.
\end{itemize}
The equivalence between \textit{i)} and \textit{ii)} is proved in Corollary 5.10 of \cite{HL09} and \textit{iii)} easily implies \textit{ii)}. Suppose now that \textit{ii)} holds and let $\epsilon, R > 0$ be such that $C(A-\epsilon I) \in \Phi$ for all $C \ge R$. Then, $C(A-\epsilon/2 \, I) = C(A-\epsilon I) + C \epsilon/2 \, I \in \Phi + \Pd^\circ \subseteq \Phi^\circ$ for all $C \ge R >0$, which is \textit{iii)}.

By means of Proposition \ref{uusc_cone} the associated elliptic cone maps $\EC{\Theta}$, $\EC{\widetilde{\Theta}}$ are constant in $\overline{\Omega}$, so $\EC{\Theta} = \EC{\Theta(\bar{x})}$ and $\EC{\ThetaD} = \EC{\ThetaD(\bar{x})}$ for any $\bar{x} \in \overline{\Omega}$. Hence, by \textit{ii)} one concludes that
\begin{align*}
\EC{\Theta}^\circ  =  \EC{\Theta(\bar{x})}^\circ = & \{ A \in \Ss(N) : \text{$\exists \, \epsilon, R > 0$
                    s.t. $C(A-\epsilon I) \in \Phi(\bar{x})$} \\
                & \quad \text{and $F(\bar{x}, C(A - \epsilon I)) \ge 0, \ \forall \ C \ge R$}  \} \\
                = & \{ A \in \EC{\Phi(\bar{x})}^\circ : \text{$\exists \, \epsilon, R > 0$ s.t. $F(\bar{x}, C(A - \epsilon I)) \ge 0, \ \forall \ C \ge R$}  \}.
\end{align*}

For $\EC{\widetilde{\Theta}}^{\circ}$, one uses \textit{iii)} and the identity $\ThetaD(\bar{x})^\circ = \widetilde{\Phi}(\bar{x})^\circ \cup \{A \in \Ss(N) : F(\bar{x}, -A) < 0 \}$, which follows from \eqref{calc_dual4}. In fact, one finds
\begin{align*}
\EC{\ThetaD}^\circ = \EC{\ThetaD(\bar{x})}^\circ = & \{ A \in \Ss(N) : \text{$\exists \epsilon, R > 0$
                    s.t. $C(A-\epsilon I) \in \widetilde{\Phi}(\bar{x})^\circ$} \\
                    & \quad \text{or $F(\bar{x}, -C(A - \epsilon I)) < 0, \ \forall \ C \ge R$}  \} \\
    = & \EC{\widetilde{\Phi}}(\bar{x})^\circ \cup \{ A \in \Ss(N) : \text{$\exists \epsilon, R > 0$ s.t. $F(\bar{x}, C(-A + \epsilon I)) < 0, \ \forall \ C \ge R$}  \}.
\end{align*}

\end{proof}

The following example gives a class of equations for which the above considerations are valid, where we denote by $\UC(\Omega, \R)$ the space of uniformly continuous real valued functions on $\Omega$.

\begin{exe}\label{exe:FGf} Let $\Phi \in \mathcal{E}, f \in \UC(\Omega, \R)$ and $G \in C(\Ss(N), \R)$ be such that the following three conditions hold:
\begin{equation}\label{GDE}
    \mbox{ $G(A ) \geq G(B) $ for each $A, B \in \Phi$ such that $A \geq B$, }
\end{equation}
\begin{equation}\label{G_strict}
\mbox{ $\forall x \in \Omega$ there exists $A \in \Phi$ such that $G(A) = f(x),$}
\end{equation}
and there exists $r^* > 0$ so that
\begin{equation}\label{GNTD}
    \mbox{ $G(A + rI) \geq G(A) + \beta(r)$ for each $r \in (0, r^*]$ and each $A \in \Phi$ }
\end{equation}
for some function $\beta: (0, r^*] \to (0, +\infty)$. Then the function $F \in C(\Omega \times \Ss(N), \R)$ defined by
$$
    F(x,A):= G(A) - f(x), \ \ x \in \Omega, A \in \Ss(N)
$$
satisfies the hypotheses of Proposition \ref{UCbranch}.

Indeed, consider the constant elliptic map taking the value $\Phi \in \mathcal{E}$, which is uniformly upper semicontinuous. Clearly $F$ is continuous and satisfies \eqref{deg_ell} with respect to $\Phi$ since $G$ satisfies \eqref{GDE}. The condition \eqref{F_strict} in this case is precisely \eqref{G_strict}. It remains only to verify that \eqref{UCF} holds for each small $\veps > 0$. Using \eqref{GNTD} one finds
\begin{align*}
F(y, A + \veps I) - F(x, A) & = G(A + \veps I) - G(A) + f(x) - f(y) \nonumber \\
    & \geq \beta(\veps) - \omega_f(|x - y|),
\end{align*}
where $\beta(\veps) > 0$ for each $\veps \in (0, r^*]$ and $\omega_f$ is a modulus of continuity for $f \in \UC(\Omega, \R)$. Hence one has \eqref{UCF} for each $\veps \in (0, r^*]$ by taking $\delta =\delta(\veps)$ so that $\omega_f(\delta) \leq \beta(\veps)$, which one can do since $\omega_f(\delta) \to 0^+$ as $\delta \to 0^+$. This completes the claim that Proposition \ref{UCbranch} applies to Example \ref{exe:FGf}.

Sufficient conditions for the requirement \eqref{G_strict} are not difficult to find. For example, since elliptic sets $\Phi$ are connected, $G(\Phi)$ is a (possibly unbounded) interval while $f(\Omega)$ is a bounded interval. Hence if $G$ is bounded from below on $\Phi$ and there exists a sequence $\{A_k\} \subset \Phi$ such that
\begin{equation}\label{G_BBC}
 \mbox{ $G\geq m > - \infty$  on $\Phi$ and  $\exists \ \{A_k\} \subset \Phi$ such that $G(A_k) \to + \infty$ as $k \to + \infty$,}
\end{equation}
then $G(\Phi) = [m, +\infty)$ for some $m \in \R$ and \eqref{GNTD} holds for each $f \in \UC(\Omega, \R)$ with $f \geq m$ on $\Omega$.

As for the needed boundary convexity, if $F$ has the form $F(x,A) = G(A) - f(x)$ and $\Phi$ is itself an elliptic cone, so that $\EC{\Phi} = \Phi$, it is possible to characterize further $\EC{\Theta}$ and $\EC{\widetilde{\Theta}}$ in terms of $G$. Let $x \in \partial \Omega$ and let $d = d_{\partial \Omega}$ be the distance function from $\partial \Omega$, which is set to be negative in $\Omega$ and positive outside. This distance function will be of class $C^2$ if $\partial \Omega$ is (see e.g. Foote \cite{Fo84}). Then, a sufficient condition for $\partial \Omega$ to be strictly $\EC{\Theta}$-convex in $x \in \Omega$ is that there exist $\alpha \in \R$ and $\epsilon >0$ such that \footnote{Here $v \otimes w$ is the {\em outer product} of $v,w \in \R^N$ which is the $N \times N$ matrix with entries $(v \otimes w)_{ij}= v_i w_j = (vw^T)_{ij}$.}
\begin{equation}\label{G_convex}
\left\{
\begin{array}{l}
\text{$G(C(D^2 d(x) + \alpha D d \otimes Dd (x) - \epsilon I)) \rightarrow +\infty$ as $C \rightarrow +\infty$} \\
\text{and $ D^2 d(x) + \alpha D d \otimes Dd (x) \in {\Phi}^\circ$}.
\end{array}
\right.
\end{equation}
Indeed, suppose that \eqref{G_convex} holds. Hence, $G(C(D^2 d(x) + \alpha D d \otimes Dd (x) - \epsilon I)) \ge f(x)$ for all $C \ge R$ and for some $\alpha \in \R$, $\epsilon, R > 0$, so $D^2 d(x) + \alpha D d \otimes Dd (x) = B \in \EC{\Theta}^{\circ}$ by Proposition \ref{natural_cones}. We now observe that $B|_{T_x \partial \Omega} = D^2 d(x)|_{T_x \partial \Omega}$ since $Dd(x) \in N_x \partial \Omega$. Hence $\partial \Omega$ is strictly $\EC{\Theta}$-convex in $x \in \Omega$, being $d$ a local defining function for $\partial \Omega$ in $x$.

Similarly, a sufficient condition for $\partial \Omega$ to be strictly $\EC{\widetilde{\Theta}}$-convex in $x \in \Omega$ is that there exist $\alpha \in \R$ and $\epsilon >0$ such that
\begin{equation*}
\left\{
\begin{array}{l}
\text{$G(C(-D^2 d(x) + \alpha D d \otimes Dd (x) + \epsilon I)) \rightarrow -\infty$ as $C \rightarrow +\infty$}, \\
\text{or $D^2 d(x) + \alpha D d \otimes Dd (x) \in {\widetilde{\Phi}}^\circ$}.
\end{array}
\right.
\end{equation*}

\end{exe}

\begin{rem}\label{SC_NTD}
For $F(x,A) = G(A) - f(x)$, we point out that the condition \eqref{GNTD} which implies the structural condition \eqref{UCF} resembles the {\em non-totally degenerate} condition of Bardi-Mannucci \cite{BaMa06} which is used (together with standard assumptions from the classical viscosity theory) to prove the validity of the comparison principle for equations which can also depend on $(u,Du)$. For equations without this dependence their condition becomes
$$
\mbox{$ \exists \eta > 0$ such that $F(x, A + rI) - F(x,A) \geq \eta r$ for each $r > 0$},
$$
which is \eqref{GNTD} for $F(x,A) = G(A) - f(x)$ and a linear $\beta(r) = \eta r$. For this reason, we will refer to \eqref{GNTD} as a non-total degeneracy condition and it gives an alternative to the standard structural condition placed on $F$ in order to obtain the comparison principle by way of the so-called Theorem on Sums (see \eqref{caba1}-\eqref{caba2} and Example \ref{SC_CAB} below).
\end{rem}

We finally recall that if the branch condition \eqref{def_branch2} holds, Proposition \ref{SHCVS} shows that $\Theta$-harmonic functions will be $\Phi$-admissible viscosity solutions of \eqref{FNE}, and therefore one obtains existence and uniqueness for the associated Dirichlet problem for an elliptic branch of the PDE. The next subsection considers come concrete examples.

\subsection{Equations involving $\lambda_k(D^2 u)$ and their perturbations.}

Concrete examples of $G \in C(\Ss(N), \R), \Phi \in \mathcal{E}$ and $f \in \UC(\Omega, \R)$ which satisfy the needed conditions \eqref{GDE}, \eqref{G_strict} and \eqref{GNTD} are provided by inhomogeneous Monge-Amp\`ere equations and the prescribed $k$-th eigenvalue equation. More precisely, for
\begin{equation}\label{NHMA}
F(x,A) = G(A) - f(x) = {\rm det}(A) - f(x), \quad f \ge 0,
\end{equation}
and
\begin{equation}\label{PkthE}
F(x,A) = G(A) - f(x) = \lambda_k(A) - f(x),\quad k=1, \ldots, N,
\end{equation}
one has the following results.

\begin{thm}\label{thm:NHMA} Let $\Omega \subset \R^N$ be a bounded and strictly convex domain with $\partial \Omega$ of class $C^2$, $f \in \UC(\Omega, \R)$ be a non-negative function, and let
$$
    \Theta_{{\rm MA}} = \Theta_{{\rm MA}}(x) := \{ A \in \Pd:  {\rm det}(A) - f(x) \geq 0 \}.
$$
Then, for each $\varphi \in C(\partial \Omega)$, there exists a unique $u \in C(\overline{\Omega})$ which is $\Theta_{{\rm MA}}$-harmonic in the sense of Definition \ref{def:Weak_SH} and such that  $u = \varphi$ on $\partial \Omega$. Moreover, $u$ is a $\Pd$-admissible viscosity solution (see Definition \ref{Vs_def}) of ${\rm det}(D^2 u(x)) - f(x) = 0$.
\end{thm}

\begin{thm}\label{thm:PkthE} Let $\Omega \subset \R^N$ be a bounded and strictly $\Pd_{k \wedge (N-k+1)}$-convex\footnote{$\Pd_{k \wedge (N-k+1)} := \{A \in \Ss(N) : \lambda_k (A) \ge 0 \text{ and } \lambda_{N-k+1} (A) \ge 0\}$, $k=1, \ldots, N$. } domain with $\partial \Omega$ of class $C^2$, $f \in \UC(\Omega, \R)$ and let
$$
    \Theta_{k} = \Theta_{k}(x) := \{ A \in \Ss(N):  \lambda_k(A) - f(x) \geq 0 \}.
$$
Then, for each $\varphi \in C(\partial \Omega)$, there exists a unique $u \in C(\overline{\Omega})$ which is $\Theta_{k}$-harmonic in the sense of Definition \ref{def:Weak_SH} and such that  $u = \varphi$ on $\partial \Omega$. Moreover, $u$ is a viscosity solution in the standard sense of $\lambda_k(D^2 u(x)) - f(x) = 0$. 
\end{thm}

\begin{proof}[Proofs of Theorems \ref{thm:NHMA} and \ref{thm:PkthE}] 

It suffices to apply the abstract theorem (Theorem \ref{thm:EU}) and Proposition \ref{SHCVS}. To do so, we follow the preceding discussion of this section. In particular, let $G$ be as in \eqref{NHMA}, then it is clearly a continuous function on $\Ss(N)$ and monotone non-decreasing on $\Phi := \Pd$. Moreover, $G$ satisfies $G \geq 0$ on $\mathcal{P}$,  $G (k\, I) \rightarrow +\infty$ as $k \rightarrow \infty$ and
\begin{equation}\label{MAnondeg}
    \mbox{ $G(A + rI) =\det(A + rI) \geq \det(A) + r^N =: G(A) + \beta(r)$ for all $r  > 0, A \in \Phi$, }
\end{equation}
so the conditions \eqref{GDE}, \eqref{G_BBC} and \eqref{GNTD} of Example \ref{exe:FGf} hold. Proposition \ref{UCbranch} then applies, guaranteeing that $\Theta_{{\rm MA}}$ extends to a uniformly upper semicontinuous elliptic map on $\overline{\Omega}$. Using \eqref{ENC1}-\eqref{ENC2} and \eqref{cone_exs}, it is easily verified that $\EC{\Theta}_{{\rm MA}} = \Pd \subset \dPd = \EC{\widetilde{\Theta}}_{{\rm MA}}$, and strict $\Pd$-convexity of $\Omega$ is equivalent to the standard notion of strict convexity, so the existence of a unique $\Theta_{{\rm MA}}$-harmonic function satisfying the boundary data follows by Theorem \ref{thm:EU}. Finally, $u$ is $\Pd$-admissible viscosity solution of the equation by a direct application of Proposition \ref{SHCVS}. Indeed, $f(x) \ge 0$ for all $x \in \Omega$, so $\partial \Theta_{{\rm MA}}(x) = \{A \in \Pd : \det(A) - f(x) = 0\}$ and \eqref{def_branch2} follows. The non-degeneracy condition \eqref{NDC} is a consequence of \eqref{MAnondeg}.

 The proof of Theorem \ref{thm:PkthE} with $F$ defined by \eqref{PkthE} can be carried out in the same way. The only difference is that one may allow $\Phi$ to be all of $\Ss(N)$, since $G(A) = \lambda_k(A)$ is monotone non-decreasing on $\Ss(N)$. This creates no problem for the ellipticity of $\Theta = \Theta_k$ in Proposition \ref{pick_branch} since $\Phi = \Ss(N)$ lacks being an elliptic set only because it is not a proper subset, but this condition is asked only to ensure that $\Theta$ defined by \eqref{def_branch1} is proper. Here $\Theta_k$ is clearly a proper subset by definition. Similarly, one easily checks that having $\Phi = \Ss(N)$ creates no problems for Propositions \ref{UCbranch} and \ref{natural_cones}. The characterization of $\EC{\Theta_k}^\circ$, $\EC{\widetilde{\Theta}_k}^\circ$ follows directly by \eqref{ENC1}-\eqref{ENC2}. Indeed, for any $\bar{x} \in \Omega$,
\begin{align*}
\EC{\Theta_k}^{\circ} & = \{ A \in \Ss(N) : \text{$\exists \epsilon, R > 0$ s.t. $\lambda_k (C(A-\epsilon I)) - f(\bar{x}) \ge 0$ for all $C \ge R$}  \}  \\
& = \{ A \in \Ss(N) : \text{$\exists \epsilon, R > 0$ s.t. $\lambda_k (A) \ge \epsilon+ {f(\bar{x})}{C^{-1}}$ for all $C \ge R$}  \} \\
& = \{ A \in \Ss(N) : \text{$\lambda_k (A) > 0$}  \}  = \Pd_k^\circ.
\end{align*}
The equality $\EC{\widetilde{\Theta}_k}^\circ = \Pd_{N-k+1}^\circ$ can be verified in an analogous way.
\end{proof}
Notice that $\Theta_k$ need not be convex, as noted in the introduction. For example, if $N=2$ then with
$$
\mbox{ $A_1 = \left[ \begin{array}{rr} -1 & 0 \\ 0 & 0 \end{array} \right]$ \quad and $A_2 = \left[ \begin{array}{rr} 0 & 0 \\ 0 & -1 \end{array} \right]$,}
$$
clearly $\lambda_2(A_k) = 0$ for $k=1,2$ while $\lambda_2(tA_1 + (1-t)A_2) < 0$ for each $t \in (0,1)$.

We next consider examples of perturbing the equation generated by \eqref{NHMA}. These examples are chosen to indicate where classical viscosity approaches may have difficulty in being applied and yet the methods developed herein are able to operate freely. 

Let $f \in \UC(\Omega, \R)$ and $M \in \UC(\Omega, \Ss(N))$ and define $F \in C(\Omega \times \Ss(N), \R)$ by
\begin{equation}\label{PNHMA}
    F(x,A) = \left[ {\rm det}(A + M(x)) \right]^{1/N} - f(x).
\end{equation}
Consider the uniformly upper semicontinuous elliptic map $\Phi: \Omega \to \mathcal{E}$ defined by
\begin{equation}\label{PMA_Phi}
\Phi(x):= \{ A \in \Ss(N): \ A + M(x) \geq 0 \} = M(x) + \mathcal{P}, \ \ x \in \Omega.
\end{equation}

\begin{prop}\label{PMAcomparison} If $f \geq 0$ on $\Omega$, then the map defined by
$$
    \Theta(x):= \{ A \in \Phi(x): \ F(x,A) =  \left[ {\rm det}(A + M(x)) \right]^{1/N} - f(x) \geq 0 \}, \ \ x \in \Omega.
$$
extends to a uniformly upper semicontinuous elliptic map on $\overline{\Omega}$. Moreover, the comparison principle between $\Theta$-subharmonic and $\Theta$-superharmonic functions holds.
\end{prop}

This proposition yields in turn the comparison principle for $\Phi$-admissible viscosity solutions of the partial differential equation defined by \eqref{PNHMA}; that is,
\begin{equation}\label{PMAeq}
     F(x,D^2u) = \left[ {\rm det}(D^2u + M(x)) \right]^{1/N} - f(x) = 0, \ \ f \geq 0,
\end{equation}
which can be verified by combining Theorem \ref{thm:CP} and Proposition \ref{SHCVS}. The corresponding existence result for equation \eqref{PMAeq} can be easily obtained in view of Theorem \ref{thm:EU}, but we omit the details.

\begin{proof}[Proof of Proposition \ref{PMAcomparison}] $\Phi$ is clearly elliptic since each $\Phi(x)$ is a translate in $\Ss(N)$ of the elliptic set $\mathcal{P}$ and the uniform continuity of $M$ yields
$$
\mbox{$A + \veps I + M(y) \geq A + M(x) \geq 0$ for each $A \in \Phi(x)$ and $x,y \in \Omega$ with $|x-y| < \delta_M(\veps)$;}
$$
that is,
$$
\mbox{$\Phi(x) + \veps I \subset \Phi(y)$ for each $x,y \in \Omega$ with $|x-y| < \delta_M(\veps)$.}
$$
Interchanging the roles of $x$ and $y$ gives the uniform upper semicontinuity of $\Phi$.

Clearly $F$ is continuous and satisfies \eqref{deg_ell} with respect to $\Phi$. For the property \eqref{F_strict}, notice that for each $x \in \Omega$ the range of
$$
    \left[ {\rm det}(A + M(x)) \right]^{1/N}: \Phi(x) \to \R
$$
is the interval $[0, +\infty)$ since $\Phi(x)$ is connected, $0$ is attained by $A = - M(x) \in \Phi(x)$, and ${\rm det}(tI + M(x)) \to +\infty$ as $t \to +\infty$. It remains only to verify that \eqref{UCF} holds (for each small $\veps > 0$). This follows from the uniform continuity of $f$ and $M$ since
\begin{align*}
F(y, A + \veps I) - F(x, A) & = \left[ {\rm det}(A + \veps I + M(y)) \right]^{\frac{1}{N}} - \left[ {\rm det}(A + M(x)) \right]^{\frac{1}{N}} + f(x) - f(y) \nonumber \\
    & \geq \left[ {\rm det}(A + \veps I + M(y)) \right]^{\frac{1}{N}}- \left[ {\rm det}(A + M(x)) \right]^{\frac{1}{N}} - \omega(|x - y|),
\end{align*}
which will be non-negative if $|x-y| < \delta(\veps)$ for a suitable $\delta$ involving the modulus of continuity of $M$. Hence $\Theta$ will be as advertised.
\end{proof}

We observe that some issues arise if one attempts to apply the classical viscosity solution methods to the equation \eqref{PMAeq}. First, we recall the classical structural condition placed on $F$ to ensure the validity of the comparison principle. The condition (3.14) of Crandall-Ishii-Lions \cite{CrIsLiPL92} (rewritten for $F(x,A)$ which is increasing in $A$) is the following: there exists $\omega : [0, \infty] \rightarrow [0, \infty]$ such that $\omega(0^+)=0$ and
\begin{equation}\label{caba1}
    F(x,A)-F(y,B) \le \omega(\alpha|x-y|^2 + |x-y|)
\end{equation}
whenever $x,y \in \Omega$ and $A,B \in \Ss(N)$ satisfy
\begin{equation}\label{caba2}
- 3 \alpha \left(\begin{array}{cc} I & 0 \\ 0 & I \end{array} \right) \le
\left(\begin{array}{cc} A & 0 \\ 0 & -B \end{array} \right) \le
3 \alpha \left(\begin{array}{cc} I & -I \\ -I & I \end{array} \right)
\end{equation}
This condition may fail in cases where \eqref{UCF} holds.

\begin{exe}\label{SC_CAB} Let $\Omega \subset \R^2$ be any smooth bounded domain containing the origin and let $F(x,A) := [\det(A + M(x))]^{1/2} - f(x)$ with $f \geq 0$ and
\[
M(x) := \left(\begin{array}{cc} |x| & 0 \\ 0 & 0 \end{array} \right) .
\]
While $F$ satisfies \eqref{UCF} by Proposition \ref{PMAcomparison}, $F$ cannot satisfy \eqref{caba1}-\eqref{caba2} on $\Omega$. Indeed, suppose that \eqref{caba1}-\eqref{caba2} holds for some $\omega : [0, \infty] \rightarrow [0, \infty]$ such that $\omega(0^+)=0$. Let $\{x_n\}_{n \in \N} \subset \Omega \setminus \{0\}$ such that $x_n \rightarrow 0$ and set
\[
A_n := \left(\begin{array}{cc} 0 & 0 \\ 0 & \frac{1}{2|x_n|} \end{array} \right), \quad B_n := \left(\begin{array}{cc} 0 & 0 \\ 0 & \frac{1}{|x_n|} \end{array} \right).
\]
We claim that $A_n, B_n$ satisfy the admissibility condition \eqref{caba2} with $3 \alpha_n = 1/|x_n|$. Indeed, $A_n \ge 0$ and $B_n \le I / |x_n| = 3 \alpha_n I$. As for the upper bound in \eqref{caba2},
\[
\langle A_n \xi, \xi \rangle - \langle B_n \eta, \eta \rangle = \frac{\xi_2^2}{2|x_n|} - \frac{\eta_2^2}{|x_n|} = \frac{1}{|x_n|} \left(\frac{\xi_2^2}{2} - \eta_2^2 \right) \le \frac{1}{|x_n|} (\xi_2 - \eta_2 )^2 \le \frac{1}{|x_n|} |\xi - \eta|^2
\]
for all $\xi, \eta \in \R^2$. Therefore,
\begin{align*}
\frac{1}{\sqrt{2}} & = \det(A_n + M(x_n))^{\frac{1}{2}} - \det(B_n + M(0))^{\frac{1}{2}} = F(x_n,A_n)-F(0,B_n) - f(x_n) + f(0) \\
    & \le \omega(\alpha_n |x_n|^2 + |x_n|) + \omega_f(|x_n|) = \omega\left(\frac{|x_n|}{3} + |x_n|\right) + \omega_f(|x_n|)  \rightarrow 0
\end{align*}
as $n \rightarrow \infty$, which is impossible.
\end{exe}

Given that the structural condition \eqref{caba1}-\eqref{caba2} may fail for \eqref{PMAeq}, classical viscosity theory can attempt to remedy this by converting \eqref{PMAeq} by way of a well chosen Bellman operator as in Ishii \cite{Is89} or Crandall-Ishii-Lions \cite{CrIsLiPL92}. This too can fail in certain cases that we can treat, as will be described in the following subsection (see Remarks \ref{PMA1} and \ref{PMA2}). On the other hand, \eqref{UCF} can fail in cases where \eqref{caba1}-\eqref{caba2} holds. For example,  this happens generically for linear equations; that is, if $F(x,A) = {\rm tr}(a(x) A) - f(x)$ and $a(x) \geq 0$ is not a constant coefficient matrix. In such a case, the natural elliptic map $\Theta$ given by $\Theta(x) = \{A \in \Ss(N): F(x,A) \geq 0 \}$ takes values in a half space with inclination that depends on $a(x)$. As $a(x)$ varies, the hyperplane boundaries of $\Theta(x)$ will be divergent as $||A|| \to \infty$ so that the uniform upper semicontinuity \eqref{uusc2} will typically fail. On the other hand, \eqref{caba1}-\eqref{caba2} merely requires sufficient regularity. In this sense, one can note that the approach of using elliptic branches developed following the path initiated by Harvey and Lawson in \cite{HL09} is truly a {\em nonlinear theory}. See also the Cautionary Note 2.7 in Section 2 of their subsequent paper \cite{HL11}. If one desires to be stubborn and use the techniques developed here to treat linear equations, this can be done by again exploiting a suitable reformulation in terms of a {\em nonlinear} Bellman operator. This too will be sketched in the next subsection (see Remark \ref{linear_eqs}, Example \ref{exe:linear} and the subsequent discussion).

\subsection{Reformulations involving Bellman operators}

As indicated above, in order to complete a comparison with classical viscosity solution methods, we will make use of reformulations involving nonlinear Bellman operators. We begin by recalling a class of such operators calibrated to the considerations we have in mind. We consider $F: \Omega \times \Ss(N)$ of the form
\begin{equation}\label{BO}
F(x,A) := \inf_{\beta \in \mathcal{B}} \{ {\rm tr} \left[ a(x, \beta) A \right] + d(x, \beta)  \},
\end{equation}
where $\mathcal{B}$ is any nonempty set,  $\Lambda:= \Omega \times \mathcal{B}$, $d : \Lambda \rightarrow \R$ and $a := \sigma^T \sigma : \Omega \times  \mathcal{B} \rightarrow \Ss(N)$ with $\sigma$ a matrix valued function on $\Lambda$. The canonical assumptions of the standard viscosity theory which ensure that the structural condition \eqref{caba1}-\eqref{caba2} holds for $F$ defined by \eqref{BO} are:
\begin{itemize}
\item[\textit{i)}] The functions $a, d$ are bounded on $\Lambda$,
\item[\textit{ii)}] The function $\sigma$ is Lipschitz continuous in $\Omega$, uniformly in $\mathcal{B}$,
\item[\textit{iii)}] The function $d$ is continuous in $\Omega$, uniformly in $\mathcal{B}$,
\end{itemize}
as one can verify by consulting Section 3 of Ishii \cite{Is89} or Section IV.1 of Ishii-Lions \cite{IsLiPL90}. 

As noted in Example \ref{SC_CAB}, the standard structural condition \eqref{caba1}-\eqref{caba2} which is used for the comparison principle may fail for the perturbed Monge-Amp\`{e}re equation \eqref{PMAeq}, which is due to the way in which $x$-dependence enters explicitly into the second order part of the operator by way of the matrix function $M = M(x)$. When $M \equiv 0$, one knows from Section V.3 of \cite{IsLiPL90} that it can be useful to rewrite the nonlinear operator as
\begin{equation}\label{BO1}
    \left[ {\rm det}(A + M(x)) \right]^{1/N} = \inf_{\beta \in \mathcal{B}} \{ {\rm tr} \left[\beta (A + M(x)) \right] \}, \ \ \mathcal{B}:= \{ \beta \in \Ss(N): \ {\rm det}(\beta) = N^{-N}  \}.
\end{equation}
This representation is valid for $A$ such that $A + M(x) \geq 0$ which is precisely the $\Phi$-admissibility condition $A \in \Phi(x)$ as given in \eqref{PMA_Phi}.

\begin{rem}\label{PMA1}
Notice that \eqref{BO1} is of the form \eqref{BO} with $a(x, \beta) = \beta$ and $d(x, \beta) = {\rm tr}[\beta M(x)]$ and hence $d$ fails to satisfy the conditions \textit{i)} and \textit{iii)} above and hence this approach does not resolve the problem.
\end{rem}

A similar problem has been encountered in Bardi-Mannucci \cite{BaMa08a} in their study of Monge-Amp\`ere type equations in a subelliptic environment such as a homogeneous Carnot group (see also \cite{BaMa13}). Their idea is to take the natural logarithm of \eqref{PMAeq} and rewrite the equation as a Bellman equation of the form
\begin{equation}\label{PMA_BM}
    \inf_{\beta \in \mathcal{B}} \mathcal{L}_{\beta} (x,A) = \log f(x)
\end{equation}
where they need only that the coefficients of the involved linear operators $\mathcal{L}_{\beta}$ possess a suitable uniform regularity with respect to $\beta$, provided that $A$ is bounded away from zero.

\begin{rem}\label{PMA2}
The term $\log f(x)$ in \eqref{PMA_BM} forces $f$ to be taken positive everywhere on $\Omega$ and we have no such restriction in Proposition \ref{PMAcomparison} which also treats the equation in its original form. 
\end{rem}

We return now to the the linear case where
\begin{equation}\label{LE}
    F(x,A) = {\rm tr}[a(x) A] - f(x).
\end{equation}
We will assume that $a = \sigma^T \sigma \in \UC(\Omega, \Ss(N))$ and $f \in \UC(\Omega, \R)$ in order to have the comparison principle. Of course such a linear equation is trivially a Bellman equation of the form \eqref{BO} with $a(x, \beta) = a(x)$ and $d(x, \beta) = -f(x)$. As such, the classical structure conditions \textit{i)}-\textit{iii)} would require only that $\sigma$ be Lipschitz on $\Omega$.

\begin{rem}\label{linear_eqs} It is known that, in general, the comparison principle may fail for a linear equation of the form ${\rm tr}[\sigma^T \sigma(x) D^2u(x)] - u = 0$ if $\sigma$ fails to be Lipschitz (see Section 3 of Ishii \cite{Is89}). In Ishii's counterexample, the two viscosity solutions have unbounded Hessians. As noted at the end of the previous subsection, the unboundedness of $A$ on natural branches $\Theta$ is also the reason that $F$ defined by \eqref{LE} will not satisfy our structure condition \eqref{UCF} in general.
\end{rem}

However, by truncating $F$ where the Hessian could be large, one can force linear examples into the present theory. We will make a truncation with respect to {\em Pucci's minimal operator}, which is defined as follows. For $\lambda, \Lambda$ with $0 < \lambda \leq \Lambda < + \infty$ set
$$
    \mathcal{B}_{\lambda, \Lambda} := \{ \beta \in \Ss(N): \lambda \leq  \lambda_1(\beta) \ldots  \leq \lambda_N(\beta) \leq \Lambda \}
$$
and define the functional $\mathcal{M}^{-}_{\lambda, \Lambda}: \Ss(N) \to \R$ by
$$
    \mathcal{M}^{-}_{\lambda, \Lambda}(A) := \inf_{\beta \in \mathcal{B}_{\lambda, \Lambda}} {\rm tr}[\beta A] = \lambda \sum_{\lambda_k(A) > 0} \lambda_k(A) - \Lambda \sum_{\lambda_k(A) < 0} \lambda_k(A).
$$

\begin{exe}\label{exe:linear}
 Let $F(x,A) = {\rm tr}[a(x) A] - f(x)$ with $f \in \UC(\Omega, \R)$ and $a \in \UC(\Omega, \Ss(N))$ a non constant map such that
\begin{equation}\label{a_conditions}
    a(x) \in  \mathcal{B}_{\lambda, \Lambda}, \ \ \text{for each} \ x \in \Omega.
\end{equation}
For each $h \in \R$ define the truncation $F_h: \Omega \times \Ss(N) \to \R$ by
$$
    F_h(x,A) := \min\{ F(x,A), \mathcal{M}^{-}_{\frac{\lambda}{2}, \Lambda}(A) + h \}.
$$
The set valued map $\Theta_h: \Omega \to \Ss(N)$ defined by
\begin{equation}\label{Thetah}
    \Theta_h(x) := \{ A \in \Ss(N): F_h(x,A) \geq 0 \}
\end{equation}
is a uniformly upper semicontinuous elliptic map, and hence the comparison principle of Theorem \ref{thm:CP} holds for $\Theta_h$.

Indeed, for each $x \in \Omega$, the subset of $\Ss(N)$ defined by \eqref{Thetah} is clearly non empty, closed and proper. It also satisfies the positivity property \eqref{elliptic2} since
for each $A \in \Theta_h(x)$ and each $P \in \mathcal{P}$ one has
$$
\mbox{$F(x, A+P) \geq F(x,A) \geq 0$ \quad and \quad $\mathcal{M}^{-}_{\frac{\lambda}{2}, \Lambda}(A+P) + h \geq \mathcal{M}^{-}_{\frac{\lambda}{2}, \Lambda}(A) + h \geq 0$,}
$$
and hence $\Theta_h$ is an elliptic map for each $h \in \R$.

    For the uniform upper semicontinuity, let $\veps > 0$ and $x \in \Omega$ be arbitrary. For each $B \in \Theta_h(x) + \veps I$ one has $B - \veps I \in \Theta_h(x)$ and hence by \eqref{Thetah} we have
\begin{equation}\label{Lexe1}
\mbox{$F(x, B - \veps I) \geq 0$ \quad and \quad $\mathcal{M}^{-}_{\frac{\lambda}{2}, \Lambda}(B - \veps I) + h \geq 0$,}
\end{equation}
and we need $B \in \Theta(y)$; that is,
\begin{equation}\label{Lexe2}
\mbox{$F(y, B) \geq 0$ \quad and \quad $\mathcal{M}^{-}_{\frac{\lambda}{2}, \Lambda}(B) + h \geq 0$,}
\end{equation}
for each $y \in \Omega$ such that $|x-y| < \delta$ for some $\delta = \delta(\veps)$ which may also depend on $h, ||f||, \lambda$ and $ \Lambda$ but is independent of $B,x$ and $y$. The second request in \eqref{Lexe2} follows from the corresponding part in \eqref{Lexe1} and the monotonicity of the Pucci operator. For the first request in \eqref{Lexe2}, without loss of generality we can restrict attention to those $(y,B) \in \Omega \times [\Theta_h(x) + \veps I]$ with $F(y,B) \leq 1$. Since $a(y) \in \mathcal{B}_{\lambda, \Lambda}$ one has
\begin{equation}\label{Lexe3}
\mathcal{M}^{-}_{\lambda, \Lambda}(B) \leq {\rm tr}[a(y)B] = F(y,B) + f(y) \leq 1 + f(y) \leq 1 + ||f||,
\end{equation}
where $f \in UC(\Omega, \R)$ is bounded. Denoting the Pucci operator by $\mathcal{M}^{-}_{\lambda, \Lambda}(B) := \lambda {\rm tr}[B^+] - \Lambda {\rm tr}[B^-]$ and using \eqref{Lexe2} - \eqref{Lexe3} we find
\begin{equation}\label{Lexe4}
\mbox{ $\lambda {\rm tr}[B^+] - \Lambda {\rm tr}[B^-] \leq 1 + ||f||$ \quad and \quad
    $- \lambda {\rm tr}[B^+] + 2 \Lambda {\rm tr}[B^-] \leq 2h$.}
\end{equation}
Linear combinations of \eqref{Lexe4} yield
$$
    \mbox{ ${\rm tr}[B^-] \leq \frac{1 + 2h + ||f||}{\Lambda}$ \quad and \quad ${\rm tr}[B^+] \leq \frac{2 + 2h + ||f||}{\Lambda}$,}
$$
and hence there exists $C = C(h, ||f||, \lambda, \Lambda)$ such that
\begin{equation}\label{Lexe5}
\mbox{ $||B|| \leq C$ for each $B \in \Theta_h(x) + \veps I$.}
\end{equation}
Using the uniform continuity of $a$ and $f$, the ellipticity bound on $a \geq \lambda I$, the bound \eqref{Lexe5} and the fact that $B \in \Theta_h(x)$, we find

\begin{align*}
F(y, & B)  =  F(y, B + \veps I - \veps I) = {\rm tr}[a(y)(B - \veps I) + \veps a(y) ] - f(y) \\
     = & \ {\rm tr}[(a(y)-a(x))(B - \veps I)] + {\rm tr}[a(x)(B - \veps I)] + {\rm tr}[\veps a(y) ] + (f(x) - f(y)) - f(x) \\
    \geq & \ \veps \lambda N - \omega_f(|x-y|) - \omega_a(|x-y|) [ C(h, ||f||, \lambda, \Lambda) + \veps N ],
\end{align*}
which is non negative if $|x-y| < \delta(\veps, h, ||f||, \lambda, \Lambda)$. This finishes the needed claim \eqref{Lexe2}.
\end{exe}

We conclude with a few final remarks about this example. The truncation $F_h$ can be realized as a Bellman operator
$$
    F_h(x,A) = \inf_{\gamma \in \mathcal{C}_{\lambda, \Lambda}} \left\{ {\rm tr}[\tilde{a}(x, \gamma) A] + d(x,\gamma) \right\}
$$
by considering the parameter space
$$
\mathcal{C}_{\lambda, \Lambda} = \{ \gamma \in \UC(\Omega, \Ss(N)): \gamma \equiv \beta \in \mathcal{B}_{\frac{\lambda}{2}, \Lambda} \ \text{or} \ \gamma = a \},
$$
where $a$ is the (non constant) matrix coefficient in \eqref{LE}, and defining
$$
    \tilde{a}(x, \gamma) = \left\{   \begin{array}{c} a(x)  \\
    \beta \end{array} \right. \ \ \text{and} \ \  d(x, \gamma) = \left\{ \begin{array}{c} -f(x)  \\
    h \end{array} \right. \quad \text{if} \quad
    \left\{ \begin{array}{l} \gamma = a \\
    \gamma \equiv \beta \end{array} \right. .
$$
For this Bellman operator, the coefficient $\tilde{a}(x,\gamma)$ fails to be $\sigma^T \sigma$ with $\sigma$ Lipschitz for the parameter $\gamma = a$ unless $a$ has this property, which we do not assume (this is the standard structure condition \textit{ii)} appearing after formula \eqref{BO}). We assume only that $a$ is $\UC(\Omega, \Ss(N))$ with the uniform ellipticity condition \eqref{a_conditions}.

Finally, since $F_h(x, \cdot)$ is concave and uniformly elliptic one has H\"{o}lder regularity theory for the viscosity solutions of $F_h(x, D^2u) = 0$ provided that $a \in C^{0, \alpha}(\overline{\Omega})$ with $\alpha \in (0,1)$ (e.g., see Chapter 8 of Caffarelli-Cabr\`e \cite{CaCa95}). In particular, viscosity solutions are classical solutions with a uniform bound on the Hessian and hence $F_h$ reduces to $F$ for each $h$ such that $h \geq \bar{h}(N, \lambda, \Lambda, \alpha)$. In this contorted way, one can force a uniformly elliptic linear equation into the intrinsically nonlinear theory initiated by Harvey and Lawson.

\bibliographystyle{plain}

\bibliography{bib_cirant_payne} 

\end{document}